\documentclass[11pt,reqno,letterpaper]{amsart}
\usepackage{amssymb}
\usepackage{graphicx,color}
\usepackage{hyperref}
\usepackage[normalem]{ulem}
\usepackage{xcolor}
\usepackage{fullpage} 
\usepackage{cases} 
\usepackage{enumitem} 
\usepackage{mathtools}
\usepackage{amssymb}
\usepackage{hyperref}
\usepackage[normalem]{ulem}
\usepackage{fullpage} 

\usepackage{cite} 

\def\Xint#1{\mathchoice
   {\XXint\displaystyle\textstyle{#1}}%
   {\XXint\textstyle\scriptstyle{#1}}%
   {\XXint\scriptstyle\scriptscriptstyle{#1}}%
   {\XXint\scriptscriptstyle\scriptscriptstyle{#1}}%
   \!\int}
\def\XXint#1#2#3{{\setbox0=\hbox{$#1{#2#3}{\int}$}
     \vcenter{\hbox{$#2#3$}}\kern-.5\wd0}}
\def\dashint{\Xint-}

\newtheorem{theorem}{Theorem}[section]
\newtheorem{corollary}[theorem]{Corollary}
\newtheorem{lemma}[theorem]{Lemma}
\newtheorem{proposition}[theorem]{Proposition}
\newtheorem{definition}[theorem]{Definition}
\newtheorem{remark}[theorem]{Remark}

\numberwithin{theorem}{section}
\numberwithin{equation}{section}

\begin{document}

\title{On the characterization of \\ polyharmonic functions through iterated means}

\author[F. Charro]{Fernando Charro$^\dagger$}
\address{Department of Mathematics, Wayne State University, 656 W. Kirby, Detroit, MI 48202, USA}
\email{fcharro@wayne.edu}
\thanks{}

\author[C. Lebiedzik]{Catherine Lebiedzik}
\address{Department of Mathematics, Wayne State University, 656 W. Kirby, Detroit, MI 48202, USA}
\email{catherine.lebiedzik@wayne.edu}
\thanks{}

\author[N. Raihen]{Md Nurul Raihen}
\address{Department of Mathematics and Statistics, University of Toledo, 2801 W. Bancroft St., Toledo, OH, 43606, USA}
\email{mraihen@wayne.edu}
\thanks{}

\keywords{Polyharmonic functions, mean-value theorems, Pizzetti formulas,  iterated means.
\\
\indent 2020 {\it Mathematics Subject Classification:} 
31B30, 
35J30, 
35B05. 
\\
\indent$\dagger$  Corresponding Author}
\date{}

\begin{abstract}
We introduce an infinite family of mean-value formulas (exact and asymptotic) given in terms of linear combinations of iterated means. We prove that the mean-value formulas in this family characterize real-valued polyharmonic functions of finite order, and that a simple algebraic condition partitions the family into equivalence classes according to the order of  polyharmonicity.
Our key results include strong converses to the mean-value properties ---locally integrable functions satisfying a mean-value property in the family are polyharmonic--- and a regularity result ---locally integrable functions satisfying a mean-value property in the family, whether exact or asymptotic,~are~smooth. 
\end{abstract}

\maketitle

\section{Introduction}

It is well-known that the classical mean-value property characterizes harmonic functions. As it turns out, an in-principle weaker statement, 
known as  the asymptotic mean-value property,  is enough to characterize harmonicity \cite{Blaschke,Privaloff}. More precisely, a real-valued function $u$ is harmonic in $\Omega\subset\mathbb{R}^n$ if and only if 
 \begin{equation}
\label{Laplaciano.formula.asintotica}
u(x)=\dashint_{B_{r} (x)} u(y)\,dy+o({r}^2)
\quad\textrm{as} \ {r} \to0
\end{equation}
for each $x \in \Omega$, where we denote 
$\dashint_{B_{r} (x)} u(y)\,dy=\frac{1}{|B_{r} (x)|}\int_{B_{r} (x)} u(y)\,dy$.
In fact, a less well-known formula due to Pizzetti, see \cite{Pizzetti.1909, O}, provides further insight into the structure of the remainder in~\eqref{Laplaciano.formula.asintotica}.
The Pizzetti ``mean-value formula" expresses the average of an arbitrary $C^{2m}$ function over a ball or sphere as a power series involving its iterated Laplacians. 
Here, and in the sequel, $m$ will denote a positive integer.
Namely, given $u\in C^{2m}$, Pizzetti's formula states that
\begin{equation}\label{Pizzetti.formula.volume}
\dashint_{B_{r}(x)}u(y)\,dy
=
u(x)+\sum_{k=1}^{m}c_{k}\,\Delta^{k}u(x)\,{r}^{2k}+o({r}^{2m})
\quad\textrm{as} \ {r} \to0,
\end{equation}
 for 
$c_k=1/\big(
2^{k} k! \prod_{j=1}^{k}(n+2j)\big)$. 
There is a corresponding formula for the spherical mean, which only differs from \eqref{Pizzetti.formula.volume} in the coefficients in the expansion, given by $d_k=(n+2k)/n\,c_{k}$; see Appendix~\ref{appendix.pizzetti.formulas}.
From   \eqref{Pizzetti.formula.volume}, 
the equivalence between the asymptotic and classical mean-value properties for harmonic functions  becomes more clear.

We say that a function $u$ is  $m$-harmonic in $\Omega\subset\mathbb{R}^n$ (or polyharmonic of order $m$) if and only if  $u\in C^{2m}(\Omega)$ and $\Delta^{m} u=0$ in $\Omega$, where $\Delta^{2}u=\Delta(\Delta u)$, $\Delta^{3}u=\Delta(\Delta^2 u)$ and so forth (see \cite{Aronszajn.et.al.1983,Gazzola.et.al.2010} for background on polyharmonic functions).

 Polyharmonic equations arise in linear elasticity, hydrodynamics, structural engineering, and, more recently, in scattered data interpolation relevant to image processing and reconstruction  \cite{Chui.2009, Gazzola.et.al.2010, Kirmani.Jamil.2018, Kobayashi.et.al.2011, Li.et.al.1999, Meleshko.2003, Schmaltz.2014}.
  Since  mean-value formulas hold under more lenient regularity conditions than the corresponding partial differential equations,
 they can provide a unified approach to polyharmonicity in diverse settings beyond the Euclidean space, such as Riemannian manifolds, Carnot groups, and~graphs.
 
We  find several different mean-value characterizations  of polyharmonicity  in the literature.
In 1909,
Pizzetti \cite{Pizzetti.1909} showed \eqref{Pizzetti.formula.volume} for the spherical mean in dimensions $n=2,3$ and deduced that an $m$-harmonic function satisfies a mean-value formula involving its iterated Laplacians up to order $m-1$. 
Sbrana \cite{Sbr} proved the converse result in dimension $n=2$. Specifically, he proved that if a function with finite, integrable derivatives of order $2m-2$ in a domain $\Omega$ satisfies Pizzetti's mean-value formula, then it is $m$-harmonic in~$\Omega$.
Nicolesco extended the Pizzetti-Sbrana mean-value property to arbitrary dimensions for the spherical and solid means \cite{Nicolesco.book.1936}. 
Given the Pizzetti-Sbrana mean-value  property involved derivatives of the solution,  subsequent research aimed to produce mean-value characterizations of polyharmonic functions that did not  involve~derivatives. 
In \cite{NM, Nicolesco.book.1936}, Nicolesco   
 gave the following characterization:
a function $u\in L_{\textnormal{loc}}^1(\Omega)$ is $m$-harmonic in $\Omega$ if and only if for almost every $x\in \Omega$ and almost every~${r}$ with $0<{r}<\textnormal{dist}(x, \partial \Omega)$ (with the understanding, here and in the sequel, that $\textnormal{dist}(x, \partial \Omega)=\infty$ if $\Omega=\mathbb{R}^n$), it holds
\[
     u(x)
     \cdot
    \begin{vmatrix}
    1 & 1 & \ldots & 1\\
    1 & \frac{n}{n+2}  & \ldots & \frac{n}{n+2(m-1)}\\
    \vdots & \vdots & \ddots & \vdots\\
    1 & \frac{n^{m-1}}{(n+2)^{m-1}}  & \ldots & \frac{n^{m-1}}{(n+2m-2)^{m-1}}
    \end{vmatrix}
=
    \begin{vmatrix}
     \mu_{0}(u,x,{r}) & 1 & \ldots & 1\\
     \mu_{1}(u,x,{r}) & \frac{n}{n+2} & \ldots & \frac{n}{n+2(m-1)}\\
     \vdots & \vdots & \ddots & \vdots\\
     \mu_{m-1}(u,x,{r})& \frac{n^{m-1}}{(n+2)^{m-1}}  & \ldots & \frac{n^{m-1}}{(n+2m-2)^{m-1}}
    \end{vmatrix}
\]
for the iterated means 
\[
\mu_{k}(u,{r},x)
=
\begin{cases}
\displaystyle\dashint_{\partial B_{r}(x)}u\,d\mathcal{H}^{n-1}(y)&\textrm{for}\ k=0\vspace{5pt}\\
\displaystyle
\dashint_{ B_{r}(x)}u\,dy=
\frac{n}{{r}^{n}}\int_{0}^{r} t^{n-1}\mu_{0}(u,t,x)\,dt
&\textrm{for}\ k=1\vspace{5pt}\\
\displaystyle\frac{n}{{r}^{n}}\int_{0}^{r} t^{n-1}\mu_{k-1}(u,t,x)\,dt&\textrm{for}\ k\geq2.
\end{cases}
\]
Cheng \cite{CM} noted and corrected an oversight in Nicolesco’s original proof, and pointed out 
that  converses of asymptotic mean-value properties require local uniformity in $r$. 
His analysis focused on the biharmonic case, with only brief remarks on extensions to higher orders.

Picone \cite{M} proved a different mean-value property for $m$-harmonic functions, in which the mean-values are considered on concentric spheres. More precisely, he proved that for  every $x\in\Omega$ and   ${r}<\textnormal{dist}(x,\partial\Omega)/\sqrt{m}$, a $m$-harmonic function $u$ satisfies
\begin{equation} \label{MVP.Picone}
    u(x)=\sum_{j=0}^{m-1}\binom{m}{j}(-1)^{j}\,\dashint_{\partial B_{{r}_j}(x)}u(y)\,d\mathcal{H}^{n-1}(y),
 \end{equation}
 where ${r}_j=\sqrt{\frac{m-j}{m}}{r}$ for $j=0,\ldots,m-1$.
In the same paper, Picone conjectured that \eqref{MVP.Picone} should be characteristic for $m$-harmonic functions.
Fichera \cite{Fichera} proved this true for biharmonic functions ($m = 2$), and claimed that the general case should follow by a direct extension of his proof.
Caramanica \cite{Caramanica.1987} extended Fichera's proof, but in this way, she did not obtain Picone's conjecture but a different mean-value property that coincides with \eqref{MVP.Picone} only for harmonic and biharmonic functions ($m = 1, 2$).
Bramble and Payne  \cite{BHP} had also obtained a mean-value formula and converse theorem; their formula was given in terms of $m\times m$ determinants involving  means over concentric spheres. 
Caramuta and Cialdea \cite{Caramuta.Cialdea.2014} proved a characterization   in the line of Bramble and Payne's under weaker hypotheses on the  converse result, which included Picone and Caramanica's formulas  as particular cases.  Further  mean-value theorems for polyharmonic functions have been obtained by  \L{}ysik \cite{Lysik.2011,Lysik.2012,Lysik.2015,Lysik.2016},  Zalcman \cite{LZ},  and others.
In recent years, a parallel line of research has extended mean-value characterizations beyond the linear setting to nonlinear equations, 
underscoring the flexibility of the mean-value approach.
Manfredi, Parviainen, and Rossi \cite{MPR} established an asymptotic mean-value characterization of $p$-harmonic functions tied to the dynamic programming principle for  random Tug-of-War games.
Additionally, in \cite{BCMR}, asymptotic mean-value formulas were obtained for a wide family of fully nonlinear second-order elliptic equations including Pucci, Isaacs, and $k$-Hessian~operators.

In contrast to the harmonic case, existing mean-value formulas for polyharmonic functions are somewhat cumbersome  
 and often lack a natural geometric or probabilistic interpretation, which limits their applicability.
For example, although it is well-known that polyharmonic functions are smooth (see \cite{Aronszajn.et.al.1983}), mean-value properties are not the standard tool to prove smoothness; instead, it is established through elliptic regularity theory, integral representation formulas, or potential theory.
Our goal is to provide a more natural and comprehensive characterization of polyharmonicity through more broadly applicable mean-value properties.

Before stating our main results in the next section, let us fix the notation used in the sequel.
In the following, we respectively denote by $A_{r}$ and $\mathcal{A}_{r}$  the integral averages over the ball and
~sphere,~i.e.,
\[
A_{r}[v](x) = \dashint_{B_{r}(x)} v(y)\, dy,
\qquad
\textrm{and}
\qquad
\mathcal{A}_{r}[v](x) = \dashint_{\partial B_{r}(x)} v(y)\,d\mathcal{H}^{n-1}(y).
\]
 We use the notation $\mathcal{H}^{k}$ for the $k$-dimensional Hausdorff measure. 
We will also consider iterated averages, given by
\[
A_{r}^m[v](x)=A_{r}\big[A_{r}^{m-1}[v]\big](x)
= \dashint_{ B_{r}(x)} A_{r}^{m-1}[v](y)\, dy
\qquad\textrm{for all}\ m\geq1,
\]
 with the convention that $A_{r}^0= I$, the identity operator, i.e.,  $A_{r}^0[v](x)=v(x)$.
Moreover, given a real polynomial $P(t)= \sum_{k=0}^{m}a_{k}t^k$, we will denote
\[
    P(A_{r})[u](x)= \sum_{k=0}^{m}a_{k}\,A_{r}^k[u](x).
\]
We define  $\mathcal{A}_{r}^m$ and  $P(\mathcal{A}_{r})$ analogously.
Whenever there is no ambiguity, we will drop the dependence on $x$ in the notation and simply write $u$, $A_{r}[u]$, $A_{r}^k[u]$, $P(A_{r})[u]$ and so forth. 
Often, we will write polynomials $P(A_{r})[u]$ as a product of  factors, which is to be understood as an operator composition,~e.g., 
\[
\begin{split}
(A_{r}^2-I)[u]
&=
\big((A_{r}+I)
(A_{r}-I)\big)[u]
\\
&=
(A_{r}+I)\big[
(A_{r}-I)[u]\big]
=
(A_{r}+I)\big[
A_{r}[u]-u\big]
=
A_{r}^2[u]-u.
\end{split}
\]
In the sequel, $\textnormal{deg(P)}$ denotes the degree of  $P$.

We will often use the `little-$o$' notation. Let us recall it here in the different forms we will use it.
 Given  a constant $c$, an integer $k$, and a real function $g$, we write 
\[
c\le g({r}) + o({r}^k)\text{ as } {r}\to 0
\qquad\big(\textrm{respectively,}\  c\ge g({r}) + o({r}^k)\text{ as } {r}\to 0\big),
\]
whenever we have
\[
\lim_{{r}\to 0} \frac{\left[ c-g({r})\right]^+}{{r}^k}=0,
\qquad\left(\textrm{resp.,}\ 
\lim_{{r}\to 0} \frac{\left[ c-g({r})\right]^-}{{r}^k}=0
\right).
\]
In particular, for functions $x\mapsto f(x,r)$ defined on a domain $\Omega\subset\mathbb{R}^n,$ we have
$f(x,r)= g(x,r) + o({r}^k)$  as~${r}\to 0$
if for every $\epsilon>0$, there exists $r_0$ such~that
\[
|f(x,r)-g(x,r)|\leq\epsilon\, {r}^k\quad\textrm{for all}\ 0<r<r_0.
\]
Note that, in principle, $r_0$ may depend on $x\in\Omega$, which needs to be taken into account when considering iterated means.
This point was first raised by Cheng \cite{CM}, who showed that converses of the higher-order asymptotic mean-value properties from Nicolesco's family require locally uniform limits.
 We show in Remark \ref{remark.uniformity.order} below
 that the same is true for our family of iterated mean-value properties;
therefore, we introduce the following clarifying definition.
\begin{definition}[{\rm $o(r^k)$ locally uniformly}] \label{loc_unif} We say
$f(x,r)=g(x,r)+o({r}^{k})$ as ${r} \to 0$ \emph{locally uniformly in $\Omega$} if for every
compact subset~$K\subset \Omega$, given $\epsilon>0$, there exists $r_0$ such that
\[
\operatorname*{ess\,sup}_{x\in K}|f(x,r)-g(x,r)|\leq\epsilon\, {r}^k\quad\textrm{for all}\ 0<r<r_0.
\]
\end{definition}
%
We are ready to state our main results, Theorems \ref{thm.MVP.polyharmonic.intro}, \ref{thm.MVP.harmonic.intro.2}, and \ref{thm.regularity} below.

\section{Statement of Main Results}

 In our first result, we introduce a family of mean-value formulas, given in terms of linear combinations of iterated means, which characterizes real-valued polyharmonic functions of finite order. 
Similarly to the harmonic case, the family includes exact and asymptotic formulas, with the latter requiring increasingly higher-order approximations as the order of polyharmonicity increases.
The result includes strong converses to the mean-value properties, in the sense that locally integrable functions satisfying a mean-value property in the family are polyharmonic.

\begin{theorem}[Mean-value characterization of polyharmonic functions]\label{thm.MVP.polyharmonic.intro}
 Let $\Omega\subset\mathbb{R}^n
$ be an open set, $m$ a positive integer, and $u \in L^1_\textnormal{loc}(\Omega)$.
 Then, the following are equivalent:
\begin{enumerate}[series=mainlist]\itemsep3pt
    \item $u$ is $m$-harmonic in $\Omega$, i.e., $u\in C^{2m}(\Omega)$ and $\Delta^mu=0$ in $\Omega$ in the classical sense.
    
      \item Let
$P$ be a real polynomial for which 1 is a root of multiplicity exactly 
$m$. Then,
\begin{equation}\label{MVP.mharmonic.turbo.intro.sharp}
    P(A_{r})[u](x)=0  \quad
    \textnormal{for a.e.}\ x\in\Omega\ \textnormal{and every}\ r<\frac{\textnormal{dist}(x,\partial\Omega)}{\textnormal{deg(P)}},
\end{equation}
  where $\textnormal{deg(P)}$ denotes the degree of  $P$. 
          \item Let
$P$ be a real polynomial for which 1 is a root of multiplicity exactly 
$m$. Then,
\begin{equation}\label{AMVP.mharmonic.turbo.intro.sharp}
    P(A_{r})[u](x)=o({r}^{2m})  \quad
    \textnormal{as}\ {r}\to0\ \textnormal{locally uniformly in}\ \Omega.
\end{equation}
\end{enumerate}
The above equivalences also hold with the spherical average $\mathcal{A}_{r}$ in place of $A_{r}$.
\end{theorem}

\begin{remark}[Optimal order of the remainders]
The order $o(r^{2m})$ in \eqref{AMVP.mharmonic.turbo.intro.sharp} is optimal. The reason is that, given a positive integer $m$ and $l\in(0,2m)$, there are smooth functions $u$ that are not  $m$-harmonic ($\Delta^mu\neq0$ in $\Omega$), but satisfy
$(A_{r}-I)^m[u](x)=o({r}^{l})$ as ${r}\to0$ locally uniformly in $\Omega$, see Corollary \ref{corollary.optimal.order}.
\end{remark}

Next, we display important particular cases of \eqref{MVP.mharmonic.turbo.intro.sharp} and \eqref{AMVP.mharmonic.turbo.intro.sharp} to be used in the sequel.

\begin{corollary}[Equivalent forms of \eqref{MVP.mharmonic.turbo.intro.sharp} and \eqref{AMVP.mharmonic.turbo.intro.sharp}]\label{remark.with.extra.equivalences}
Under the same hypotheses of Theorem \ref{thm.MVP.polyharmonic.intro}, the following are equivalent to the statements in the theorem:
 \begin{enumerate}[resume=mainlist]
    \item ``Canonical" class representative:
    For almost every $x\in\Omega$ and every   ${r}<\textnormal{dist}(x,\partial\Omega)/m$, it holds that
\begin{equation}\label{MVP.mharmonic.main.intro}
    (A_{r}-I)^m[u](x)=0,
    \end{equation} 
    or, equivalently,
\begin{equation}\label{MVP.mharmonic.main.v2}
    u(x)=\sum_{j=1}^{m}\binom{m}{j}(-1)^{j-1}A_{r}^j[u](x).
    \end{equation}
\end{enumerate}
\begin{enumerate}[resume=mainlist]    
    \item Given $Q$, a polynomial with $Q(1)\neq0$, it holds that
\begin{equation}\label{MVP.mharmonic.turbo.intro}
    \Big(Q(A_{r})\,(A_{r}-I)^m\Big)[u](x)=0,
    \end{equation}
      for almost every $x\in\Omega$ and every   ${r}<\textnormal{dist}(x,\partial\Omega)/\big(\textnormal{deg(Q)}+m\big)$, and $\textnormal{deg(Q)}$  the degree~of~$Q$. 
      \item 
     Given $Q$, a polynomial with $Q(1)\neq0$,  it holds that  \begin{equation}\label{AMVP.mharmonic.turbo.intro}
    \Big(Q(A_{r})\,(A_{r}-I)^m\Big)[u](x)=o({r}^{2m}),
    \end{equation}
as ${r}\to0$ locally uniformly in $\Omega$.  In particular, \eqref{MVP.mharmonic.main.intro} holds with remainder $o({r}^{2m})$ as $r\to0.$
\end{enumerate}
As before, these equivalences also hold with the spherical average $\mathcal{A}_{r}$ in place of $A_{r}$.
\end{corollary}

Before proceeding with our other two main results, Theorems \ref{thm.MVP.harmonic.intro.2} and \ref{thm.regularity}, let us illustrate Theorem~\ref{thm.MVP.polyharmonic.intro} and discuss some concrete instances of the family of mean-value properties.
Indeed, by taking $m=1$ in \eqref{MVP.mharmonic.main.intro}--\eqref{MVP.mharmonic.main.v2}, we identify the  classical mean-value property for harmonic functions, $u(x)=A_{r}[u](x)$. 
The asymptotic mean-value formula for harmonic functions, $u(x)=A_{r}[u](x)+o(r^2)$ as $r\to0$, emerges from \eqref{AMVP.mharmonic.turbo.intro} when both $Q$ and $m$ are set equal to 1.
Of course, there are corresponding   results for the spherical average $\mathcal{A}_{r}$, which we omit from this discussion for conciseness.
Theorem~\ref{thm.MVP.polyharmonic.intro} weaves these well-established results into a new, infinite family of mean-value formulas characterizing harmonicity, given by \eqref{MVP.mharmonic.turbo.intro}--\eqref{AMVP.mharmonic.turbo.intro}. 
For example,  all the mean-value properties of the form 
\[
u(x)
=
A_r^k[u](x)
=
\dashint_{B_{r}(x)}\dashint_{B_{r}(y_1)}
\!\!\cdots\;
\dashint_{B_{r}(y_{k-1})}u(y_k) \,dy_k\cdots dy_2\,dy_1,
\qquad k=1,2,\ldots,
\]
 with or without a remainder,
characterize harmonicity.
These mean-value properties correspond to polynomials of the form
$
Q(A_{r})=
I+A_{r}+A_{r}^2+\cdots+A_{r}^{k-1}$
in \eqref{MVP.mharmonic.turbo.intro}, \eqref{AMVP.mharmonic.turbo.intro},
so that 
\[
(A_r^k-I)[u]
=
\big(I+A_{r}+A_{r}^2+\cdots+A_{r}^{k-1}\big)
(A_r-I)[u].
\]
Other choices of $Q$ allow us to consider other examples; for instance, the choice $Q(A_r)=A_r$
produces the mean-value formula
$(A_{r}^2-A_{r})[u]=0$, or, 
\[
    \dashint_{B_{r}(x)}u(y)\,dy
    =
    \dashint_{B_{r}(x)}
    \dashint_{B_{r}(y)} u(z) \,dz
    dy,
\]
which also characterizes harmonicity. 
Let us emphasize that the most interesting part of this result is what in the literature is known  as a strong converse to the mean-value property; i.e., that locally integrable functions satisfying the above mean-value formulas are indeed harmonic.

Naturally, one may wonder why formulas \eqref{MVP.mharmonic.turbo.intro}--\eqref{AMVP.mharmonic.turbo.intro} exclude $A_{r} -I$ as a factor of $Q(A_{r})$;
especially since, due to \eqref{MVP.mharmonic.main.intro}, an $m$-harmonic function $u$ satisfies
\[
(A_{r}-I)^{m+k}[u]
=
(A_{r}-I)^{k}\big[(A_{r}-I)^{m}[u]\big]
=0
\quad\textrm{for all $k\geq1$}
\]
in $\Omega$ for every ${r}$ small enough (and similarly with remainders).
In other words, one may wonder why the converse mean-value property seems to fail when $Q(1)=0$ in \eqref{MVP.mharmonic.turbo.intro}--\eqref{AMVP.mharmonic.turbo.intro}.
 The reason is that, for every $k$, the operator $(A_{r}-I)^{m+k}$ characterizes $(m+k)$-harmonic functions (such that $\Delta^{m+k} u =0$), and not just $m$-harmonic functions. 
For example, the mean-value formula
$(A_{r}-I)^2[u]=0$, or
\begin{equation}\label{mvp.biharmonic.intro}
u(x)=2\,\dashint_{B_{r}(x)}u(y) \,dy
    -\dashint_{B_{r}(x)}
    \dashint_{B_{r}(y)} u(z) \,dz dy,
\end{equation}
characterizes the set of biharmonic functions,  which strictly contains the set of harmonic functions.
Geometrically, formula \eqref{mvp.biharmonic.intro} shows that, at every point, the difference between the value of a biharmonic function and its average must be the same as the difference between the average and its second average, i.e.,
\[
(A_{r}-I)[u](x)=(A_{r}^2-A_{r})[u](x)
\]
for all $x\in\Omega$ and ${r}$ small enough.
For a general $m$-harmonic function $u$, equation
\eqref{MVP.mharmonic.main.intro} says that
the value of $(A_{r}-I)^{m-1}[u]$ and its average must coincide.
Incidentally, one can also rewrite \eqref{mvp.biharmonic.intro}~as
\begin{equation}\label{intriguing.MVP.biharmonic}
\dashint_{B_{r}(x)}u(y) \,dy
=
\frac{u(x)}{2}
+
\frac12\,
    \dashint_{B_{r}(x)}
    \dashint_{B_{r}(y)} u(z) \,dz dy,
\end{equation}
where we see that for biharmonic functions the average at a given point must equal the arithmetic mean of the value at the center and the second average.
However, the nice geometric interpretation given by \eqref{intriguing.MVP.biharmonic} does not have 
an obvious generalization to higher orders.

\medskip

In the following result, we show that Theorem \ref{thm.MVP.polyharmonic.intro} is optimal since the polynomials of the form \eqref{MVP.mharmonic.turbo.intro}--\eqref{AMVP.mharmonic.turbo.intro} are the only ones that characterize $m$-harmonic functions via a mean-value property. 
In this way, Theorem \ref{thm.MVP.polyharmonic.intro} partitions the set of mean-value properties of the form 
$P(A_{r})[u]=0$
into equivalence classes according to the order of the polyharmonic functions they characterize.

\begin{theorem}[Optimality of the characterization]\label{thm.MVP.harmonic.intro.2}
Consider a non-constant real polynomial $P$, and assume the existence of a converse ``asymptotic mean-value property of order $2m$," given by $P$, that characterizes $m$-harmonic functions. That is to say,
 assume that
 every nontrivial  $u\in L^1_\textnormal{loc}(\Omega)$ that satisfies the ``asymptotic mean-value formula" \begin{equation}\label{AMVP.m_harmonic.gen.P.intro}
    P(A_{r})[u](x)=o({r}^{2m})  \quad
    \textrm{as}\ {r}\to0\ \textrm{locally uniformly in}\ \Omega,
\end{equation}
is $m$-harmonic in $\Omega$, but not necessarily  $(m-1)$-harmonic.
Then,  $P(A_{r})=Q(A_{r})(A_{r}-I)^m$ for some real polynomial $Q$ with $Q(1)\neq0$.
In particular, the conclusion holds when we replace \eqref{AMVP.m_harmonic.gen.P.intro}~by
    \begin{equation}\label{MVP.harmonic.gen.P.intro}
    P(A_{r})[u](x)=0  \quad
    \textrm{for a.e. $x\in\Omega$ and every ${r}<\frac{\textnormal{dist}(x,\partial\Omega)}{\textnormal{deg(P)}}$.
    }
    \end{equation}
\end{theorem}

\begin{remark}[Local uniformity of the remainders]\label{remark.uniformity.order}  
Let us  show that the local uniformity of the remainders in Theorems \ref{thm.MVP.polyharmonic.intro} and \ref{thm.MVP.harmonic.intro.2} is necessary. We adapt a counterexample by Cheng \cite{CM}.
Let 
$\Omega = (-1,1)\times(-1,1)\subset\mathbb{R}^2,$ and define
\[
u(x,y) =
\begin{cases}
x^2 + y^2 &  \textrm{in}\ \Omega^+=\{-1<x<1,\ 0\le y<1\},\\[2mm]
x^2 - y^2 & \textrm{in}\ \Omega^-=\{-1<x<1,\ -1<y<0\}.
\end{cases}
\]
Clearly, $u$ is continuously differentiable in $\Omega$, biharmonic in $\Omega^+$, and harmonic (and hence biharmonic) in $\Omega^-$.  
However, the second derivatives have a jump across the horizontal axis, so $u$ is not biharmonic in the whole of $\Omega$.
Now, for points $(x,0)$ on the horizontal axis, Cheng \cite{CM} shows~that
$A_r[u](x,0) = x^2 + r^2/4.$
Using \eqref{pizzetti.exact} we can average once more and find
$A_r^2[u](x,0)
= A_r\big(A_r[u]\big)(x,0)
= x^2 + r^2/2.$
Hence, for all $x\in(-1,1)$ and all ${r}<\textnormal{dist}\big((x,0),\partial\Omega\big)/2$, we~have
\begin{equation}\label{counterex.exact.MVP}
(A_r-I)^2[u](x,0)
=
A_r^2[u](x,0) - 2A_r[u](x,0) + u(x,0)
= 0.
\end{equation}
A similar calculation using \eqref{pizzetti.exact}
shows that \eqref{counterex.exact.MVP} also holds for all  $(x,y)\in\Omega$ and $r$ such that $B_{2r}(x,y)$ is fully contained in one of the half-squares $\Omega^\pm$, or, equivalently, whenever
\begin{equation}\label{non.uniform.r0}
r<r_0=\frac{\min\left\{|y|,\textnormal{dist}\big((x,y),\partial\Omega\big)\right\}}{2}.
\end{equation}
Thus, \eqref{counterex.exact.MVP} is valid pointwise throughout $\Omega$ for $r$ small enough and, in particular, 
\begin{equation}\label{counterex.asymptotic.MVP}
(A_r-I)^2[u](x,y)
=
o(r^4)\quad\textrm{as $r\to0$ in $\Omega$},
\end{equation}
even though $u$ is not biharmonic in $\Omega$. This is in agreement with Theorem \ref{thm.MVP.polyharmonic.intro}, since 
\eqref{non.uniform.r0} shows that the little-$o$ in \eqref{counterex.asymptotic.MVP} is not locally uniform as required in the hypotheses.
Moreover, for points $(x,y)$ and radii $r$ such that $B_{2r}(x,y)$ intersects both $\Omega^+$ and $\Omega^-$, the expression $(A_r-I)^2[u](x,y)$ does not vanish.
Hence, although \eqref{counterex.exact.MVP} holds for all $(x,y)\in\Omega$ and all sufficiently small radii (given by \eqref{non.uniform.r0}), it does not  hold for all ${r}<\textnormal{dist}\big((x,y),\partial\Omega\big)/2$ if the point is too close to the horizontal axis, which violates the hypothesis in the theorem.
This somewhat global character of the mean-value property is an interesting feature of polyharmonic functions that is not seen in the harmonic~case. 
\end{remark}

To conclude this section, we want to emphasize a regularity result of independent interest derived from our mean-value properties in Theorem \ref{thm.MVP.polyharmonic.intro}. Namely, we show
that locally integrable functions satisfying an asymptotic mean-value property of the form $P(A_{r})[u]=o({r}^{2m})$ (with 1 a root of $P$) are smooth. In particular, by Theorem \ref{thm.MVP.polyharmonic.intro}, this implies that $m$-harmonic functions are smooth, offering a mean-value-property proof of a fact  usually established with other, less elementary tools; see Remark~\ref{remark.about.regularity} below.
\begin{theorem}[Smoothness from mean-value formulas]\label{thm.regularity}
    Let  $P$  be a non-zero polynomial with $P(1)=0$ and assume that
   $u\in L^1_\textnormal{loc} (\Omega)$  satisfies the asymptotic mean-value property \eqref{AMVP.m_harmonic.gen.P.intro}.
Then $u\in\ C^\infty(\Omega)$.
In particular, the conclusion holds when we replace \eqref{AMVP.m_harmonic.gen.P.intro} by
    \eqref{MVP.harmonic.gen.P.intro}.
\end{theorem}

The rest of the paper is organized as follows.
In Section \ref{section.smooth.case}, we prove Theorems \ref{thm.MVP.polyharmonic.intro} and \ref{thm.MVP.harmonic.intro.2} in the smooth case. 
In Section \ref{section.equivalence}, we prove  an equivalence relation between mean-value formulas satisfied by locally integrable functions. 
This allows us to choose the most convenient class representative in 
all subsequent proofs, and enables the passage from smooth to locally integrable functions.
In~Section~\ref{section.regularity}
we prove Theorem \ref{thm.regularity}, i.e., that locally integrable functions that satisfy a mean-value property are smooth.
Finally, in Section \ref{section.general.case} we prove Theorems \ref{thm.MVP.polyharmonic.intro} and \ref{thm.MVP.harmonic.intro.2} in the general case of locally integrable functions.


\section{Proof of Theorems \ref{thm.MVP.polyharmonic.intro} and \ref{thm.MVP.harmonic.intro.2} in the smooth case}\label{section.smooth.case}

Let us first prove Theorem \ref{thm.MVP.polyharmonic.intro} under the additional assumption that $u$ is smooth. We start by proving the  smooth version of Corollary \ref{remark.with.extra.equivalences}.

\begin{theorem}\label{thm.mvp.mharmonic.smooth}
Let $\Omega\subset\mathbb{R}^n
$ be an open set and $u\in C^\infty(\Omega)$. The following are equivalent:
\begin{enumerate}\itemsep3pt
        \item\label{theorem.smooth.harmonic} $u$ is $m$-harmonic in $\Omega$, i.e., $\Delta^mu=0$ in $\Omega$ in the classical sense.

    \item\label{theorem.smooth.simple.MVP} For  every $x\in\Omega$ and    ${r}<\textnormal{dist}(x,\partial\Omega)/m$, it holds that
\begin{equation}\label{MVP.mharmonic.main.smooth}
    (A_{r}-I)^m[u](x)=0,
    \end{equation} 
    or, equivalently,
\begin{equation*}
    u(x)=\sum_{j=1}^{m}\binom{m}{j}(-1)^{j-1}A_{r}^j[u](x).
    \end{equation*}
    
    \item\label{theorem.smooth.poly.MVP} Given $Q$, a polynomial with $Q(1)\neq0$, it holds that
\begin{equation}\label{MVP.mharmonic.turbo.smooth}
    \Big(Q(A_{r})\,(A_{r}-I)^m\Big)[u](x)=0,
    \end{equation}
      for  every $x\in\Omega$ and    ${r}<\textnormal{dist}(x,\partial\Omega)/\big(\textnormal{deg(Q)}+m\big)$, where $\textnormal{deg(Q)}$ denotes the degree of  $Q$. 
     
      \item\label{theorem.smooth.poly.AMVP}
     Given $Q$, a polynomial with $Q(1)\neq0$, for  every $x\in\Omega$ it holds that 
\begin{equation*}
    \Big(Q(A_{r})\,(A_{r}-I)^m\Big)[u](x)=o({r}^{2m})
    \end{equation*}
as ${r}\to0$ locally uniformly in $\Omega$. In particular, \eqref{MVP.mharmonic.main.smooth} holds with remainder $o({r}^{2m})$ as $r\to0.$

\end{enumerate}
The above equivalences also hold with the spherical average $\mathcal{A}_{r}$ in place of $A_{r}$.
\end{theorem}

\begin{remark}\label{remark.about.regularity}
Smoothness of $m$-harmonic functions is typically established via elliptic regularity theory, potential-theoretic representations, or Green’s function methods;  see, for instance, \cite{Aronszajn.et.al.1983}. We choose a different approach: we will derive the smoothness of $m$-harmonic functions directly from the mean-value formulas in  Section \ref{section.regularity}; see Theorem \ref{thm.regularity}.
\end{remark}

To prove the equivalences of the statements in Theorem \ref{thm.mvp.mharmonic.smooth}, we will first start by showing that  \eqref{theorem.smooth.harmonic} implies~\eqref{theorem.smooth.simple.MVP} in Proposition \ref{thm.mvp.mharmonic.smooth.second.piece} below. Then, 
when we  assume ${r}<\textnormal{dist}(x,\partial\Omega)/\big(\textnormal{deg(Q)}+m\big)$,
all iterated averages $A_{r}^j$ with $j\leq\textnormal{deg(Q)}+m$ are well-defined, and 
we can apply the operator $Q(A_{r})$ to both sides of \eqref{MVP.mharmonic.main.smooth} to get \eqref{MVP.mharmonic.turbo.smooth}, proving that statement \eqref{theorem.smooth.simple.MVP} implies \eqref{theorem.smooth.poly.MVP}. In turn, statement \eqref{theorem.smooth.poly.MVP} immediately implies  \eqref{theorem.smooth.poly.AMVP}. Finally, we will prove that \eqref{theorem.smooth.poly.AMVP}
implies \eqref{theorem.smooth.harmonic} in Proposition~\ref{thm.mvp.mharmonic.smooth.first.piece} below.

The arguments below hold verbatim with the spherical average $\mathcal{A}_{r}$ in place of $A_{r}$.
The reason is that the key element of the arguments is Pizzetti's formula, and the formulas for the ball and the sphere, respectively, \eqref{2.series.expansion.mvp.ball} and \eqref{2.series.expansion.mvp.surface}, only differ on the coefficients, which are irrelevant to our computations.

Let show that statement
\eqref{theorem.smooth.harmonic} implies \eqref{theorem.smooth.simple.MVP}.

\begin{proposition}[\eqref{theorem.smooth.harmonic}$\Longrightarrow$\eqref{theorem.smooth.simple.MVP}]\label{thm.mvp.mharmonic.smooth.second.piece}
    Let $\Omega\subset\mathbb{R}^n
$ be open and $u\in C^\infty(\Omega)$ be $m$-harmonic in $\Omega$.
    Then, \eqref{MVP.mharmonic.main.smooth} holds for  every $x\in\Omega$ and    ${r}<\textnormal{dist}(x,\partial\Omega)/m$.
\end{proposition}

\begin{proof}
 We proceed by strong induction.
 The base case, $m=1$,  is the  well-known mean-value property for harmonic functions, which can be proved in many  ways.
One possibility is to use the fact that $\Delta u=0$ along with Pizzetti's formula \eqref{pizzetti.exact} to get
$(A_{r}-I)[u](x)=0$
for all ${r}<\textnormal{dist}(x,\partial\Omega)$. 

 As a strong induction hypothesis, assume that for every $1\leq k\leq m-1$ it holds:
 \begin{equation}\label{ind.ass}
 \begin{split}
 \textrm{Whenever $v\in C^\infty(\Omega)$ is $k$-harmonic  in $\Omega$, then}\qquad\\
 (A_{r}-I)^{k}[v](x)=0
 \textrm{ for  all $x\in\Omega$ and    ${r}<\textnormal{dist}(x,\partial\Omega)/k$.}
 \end{split}
 \end{equation}
 We wish to show that this assumption implies: If $\Delta^m u = 0$ then $(A_r-I)^m[u](x) = 0$. 
 
 Consider $x\in\Omega$ such that  ${r}<\textnormal{dist}(x,\partial\Omega)/m$. 
 Since $\Delta^m u = 0$, Pizzetti's formula \eqref{pizzetti.exact} has finitely many terms:
 \begin{equation}\label{mvp.mlap.basic.proof.1}
 (A_r - I)[u] = c_1r^2\Delta u + c_2r^4
 \Delta^2 u + \cdots + c_{m-1}r^{2m-2}\Delta^{m-1} u.
 \end{equation}
 Moreover, if $\Delta^m u = 0$, then 
 $\Delta^{m-k} u$ is  $k$-harmonic for every $1\leq k\leq m-1$,
 which, by the strong induction hypothesis \eqref{ind.ass} gives
 \[
(A_{r}-I)^{k}\big[\Delta^{m-k} u\big](x)=0
\quad\textrm{for every $1\leq k\leq m-1$.} 
 \]
 Using this in combination with \eqref{mvp.mlap.basic.proof.1} implies 
\[
\begin{split}
(A_{r}-I)^{m}[u](x)
&=(A_{r}-I)^{m-1}\big[(A_{r}-I)[u]\big](x)\\
&=c_{1}{r}^{2}(A_{r}-I)^{m-1}[\Delta u](x) +c_{2}{r}^{4}(A_{r}-I)(A_{r}-I)^{m-2}[\Delta^{2}u](x)
\\
&\hspace{65pt} +\cdots+c_{m-1}{r}^{2m-2}(A_{r}-I)^{m-2}(A_r - I)[\Delta^{m-1}u](x)=0,
\end{split}
\]
 which completes the inductive step.
\end{proof}

Next, we show that  \eqref{theorem.smooth.poly.AMVP} implies \eqref{theorem.smooth.harmonic}.
We will need the following two lemmas.

\begin{lemma} \label{epsilon.lemma1}
Consider an open set $\Omega\subset\mathbb{R}^n$ and
    let $k \ge 0$.  Assume $v(x,r)=o({r}^{k})$  as ${r} \to 0$ locally uniformly in $\Omega$. Then, for every real, nonzero polynomial $P$ of degree $p\geq1$, we have that
\begin{equation}\label{P(A).little.o}
P(A_{r})[v](x)=o({r}^{k})\quad \textrm{as ${r} \to 0$ locally uniformly in $\Omega$.}
\end{equation}
In particular,
    
    \begin{enumerate}\itemsep3pt
        
        \item[(a)] $A_{r}^{j}[v](x)=o({r}^{k})$
         as ${r} \to 0$
         locally uniformly in $\Omega$
        for all $j \ge 1$.
        
        \item[(b)] $\big(Q(A_{r})(A_{r}-I)^{j}\big)[v](x)=o({r}^{k})$  as ${r} \to 0$ locally uniformly in $\Omega$, for all $j \ge 1$, and  every polynomial $Q$ with $Q(1) \ne 0$.
        
    \end{enumerate}
\end{lemma}

\begin{proof}
Let us first prove statement (a). We proceed by induction.
Let $K$ be a compact subset of $\Omega$, and let $\epsilon>0$. By Definition \ref{loc_unif}, there exists ${r}_0$  such that 
\begin{equation*}
-\epsilon\,{r}^{k}
<v(x,r)
<\epsilon\,{r}^{k}
\end{equation*}
for every ${r}<{r}_0$ and all $x\in K$.
 Since $r_0$, and hence $r$, are independent of the point $x$, we can integrate over a ball  and deduce 
\begin{equation*}
-\epsilon\,{r}^{k}
<
A_{r}[v](x)
<\epsilon\,{r}^{k}
\end{equation*}
for every ${r}<{r}_0$ and all $x\in K$,
which proves statement (a) in the base case $j=1$.
For the induction hypothesis, assume that  there exists ${r}_1$  such that 
\[
-\epsilon\,{r}^{k}
<
A_{r}^{j-1}[v](x)<\epsilon\,{r}^{k}
\]
for every ${r}<{r}_1$ and all $x\in K$.
Integrating this inequality over a ball, and using that $A_{r}^{j}[v](x)=A_{r}\big[A_{r}^{j-1}[v]\big](x)$, we find
\[
-\epsilon\,{r}^{k}
<
A_{r}^{j}[v](x)<\epsilon\,{r}^{k}
\]
for every ${r}<{r}_1$ and all $x\in K$,
 which completes the proof of statement (a).

Let us now prove \eqref{P(A).little.o}, which in turn implies statement (b).
We can write
\[
    P(A_{r})[v](x)= \sum_{j=0}^{p}a_{j}\,A_{r}^j[v](x)
\]
for some $a_0,a_1,\ldots a_p\in\mathbb{R}$. 
As before, let $K$ be a compact subset of $\Omega$, and let $\epsilon>0$.
 By statement~(a), for each $j \ge 1$,
 there exist  ${r}_j$  such that 
$|A_{r}^{j}[v](x)|<\epsilon\,{r}^{k}$
for every ${r}<{r}_j$ and all $x\in K$. Take $\tilde{r}=\min\{r_0,r_1,\ldots,r_p\}$.
 We deduce
\[
 \big|P(A_{r})[v](x)\big|
 \leq
 \sum_{j=0}^{p}|a_{j}|\cdot\big|A_{r}^j[v](x)\big|
<\epsilon\, {r}^{k}\,\sum_{j=0}^{p}|a_{j}|
\]
for every ${r}<\tilde{r}$ and all $x\in K$,
which completes the proof of Lemma \ref{epsilon.lemma1}.
\end{proof}

\begin{lemma} \label{epsilon.lemma2}
Let $\Omega\subset\mathbb{R}^n$ be open, and $j$ a positive integer.
Then, for all $v \in C^{\infty}(\Omega)$, we have 
    \[
(A_{r}-I)^{j}[v](x)
=\frac{\Delta^{j}v(x)}{2^j(n+2)^j} \;r^{2j}+o({r}^{2j})
\]
as ${r} \to 0$ locally uniformly in $\Omega$.
In particular,
    \[
(A_{r}-I)^{j}[v](x)
=o({r}^{2j-2})
    \]
as ${r} \to 0$ locally uniformly in $\Omega$.
\end{lemma}

\begin{proof}
 We proceed by induction. The base case, 
$j=1$, follows from Lemma \ref{appendix.lemma.pizzetti}, i.e., for any  $v \in C^{\infty}(\Omega)$,
it holds that
\begin{equation*}
(A_{r}-I)[v](x)
=c_{1}\,{r}^{2}\Delta v(x)+ o({r}^{2})
\end{equation*}
as $r\to0$ locally uniformly in $\Omega$,
with $c_1
=
1/\big(2\,(n+2)\big).
$
For the inductive hypothesis, let us assume that
for every  $v \in C^{\infty}(\Omega)$, we have
\begin{equation*}
\begin{split}
(A_{r}-I)^{j-1}[v](x)
=c_{1}^{j-1}{r}^{2j-2}\Delta^{j-1}v(x)+o({r}^{2j-2})
\end{split}
\end{equation*}
as ${r} \to 0$ locally uniformly in $\Omega$. In particular, we can take
\[
u(x)
=
\sum_{k=1}^{j}c_k\,\Delta^{k}v(x)\,r^{2k},
\]
which, by Lemma \ref{appendix.lemma.pizzetti},  has the important property that
\[
(A_{r}-I)[v](x)=u(x)+o(r^{2j})
\]
as ${r} \to 0$ locally uniformly in $\Omega$.
Therefore, using  the induction hypothesis and Lemma \ref{epsilon.lemma1} to handle the remainder, we obtain
\begin{equation*}
 \begin{split}
(A_{r}-I)^{j}[v](x)&=
(A_{r}-I)^{j-1}[u](x)+o(r^{2j})\\
&=
\sum_{k=1}^{j}c_k\,r^{2k}\,(A_{r}-I)^{j-1}\big[\Delta^{k}v\big](x)+o(r^{2j})\\
&=
c_{1}^{j}{r}^{2j}\Delta^{j}v(x)
+
\sum_{k=2}^{j}c_k\,
c_{1}^{j-1}{r}^{2k+2j-2}\Delta^{k+j-1}v(x)+o({r}^{2j})
\\
&=c_{1}^{j}{r}^{2j}\Delta^{j}v(x)+o({r}^{2j}),
\end{split}
\end{equation*}
as ${r} \to 0$ locally uniformly in $\Omega$,
which completes the proof.
\end{proof}

\begin{proposition}[\eqref{theorem.smooth.poly.AMVP}$\Longrightarrow$\eqref{theorem.smooth.harmonic}]\label{thm.mvp.mharmonic.smooth.first.piece}
    Let $\Omega\subset\mathbb{R}^n
$ be open, and $Q$ a polynomial with $Q(1)\neq0$. Let $m$ be a positive integer. Then, whenever
$u\in C^\infty(\Omega)$
is such that
\begin{equation}\label{AMVP.mharmonic.turbo.smooth.propo}
    \Big(Q(A_{r})\,(A_{r}-I)^m\Big)[u](x)=o({r}^{2m})
    \end{equation}
as ${r}\to0$ locally uniformly in $\Omega$,
 the function $u$ must be $m$-harmonic in $\Omega$, i.e., $\Delta^mu=0$ in $\Omega$.
\end{proposition}

\begin{proof}
 We proceed by induction. Let us prove first the base case, $m=1$. We assume that $u\in C^\infty(\Omega)$ satisfies
\begin{equation}\label{Q.AMVP.m=1}
    \Big(Q(A_{r})\,(A_{r}-I)\Big)[u](x)=o({r}^{2})
    \end{equation}
as ${r}\to0$ locally uniformly in $\Omega$, and intend to show that $\Delta u=0$ in $\Omega$.
Consider $x\in\Omega$ and ${r}$~small enough so that the integrals in the left-hand side of \eqref{Q.AMVP.m=1} are all well-defined.
Applying the Pizzetti formula with locally uniform remainder in Lemma \ref{appendix.lemma.pizzetti} and Lemma \ref{epsilon.lemma1}, we~get
\[
    o({r}^2)
    =
    Q(A_{r})\big[A_{r}[u]-u\big] (x)
=
c_{1}{r}^{2}\,\Big(
Q(A_{r})[\Delta u] (x)
+o(1)
\Big).
\]
Let us observe that $A_{r}^j[\Delta u](x)\to \Delta u(x)$ as ${r}\to0$ for any integer $j\geq 1$. Therefore,
\[
Q(A_{r})[\Delta u] (x)
\to
Q(1)\cdot\Delta u(x)\quad\textrm{as}\ {r}\to0,
\]
where $Q(1)\neq0$ is the sum of the coefficients of $Q$. Hence,  $u$ is harmonic.

Next, we assume the following induction hypothesis:  whenever  $v\in C^\infty(\Omega)$ is such that
\begin{equation}\label{induction.hypothesis}
    \Big(Q(A_{r})\,(A_{r}-I)^{m-1}\Big)[v](x)=o\big({r}^{2m-2}\big)
    \end{equation}
as ${r}\to0$ locally uniformly in $\Omega$, then $v$ is $(m-1)$-harmonic in $\Omega$. 
Furthermore, we assume that
$u\in C^\infty(\Omega)$ satisfies \eqref{AMVP.mharmonic.turbo.smooth.propo}.  We need to show that these imply $\Delta^m u = 0$ in $\Omega$.

Let $x\in\Omega$ and assume that  ${r}<\textnormal{dist}(x,\partial\Omega)/\big(\textnormal{deg(Q)}+m\big)$
so that all iterated averages $A_{r}^j$  with $j\leq\textnormal{deg(Q)}+m$ (and hence the left-hand side of \eqref{AMVP.mharmonic.turbo.smooth.propo}) are well-defined.
From the locally uniform Pizzetti formula in Lemma \ref{appendix.lemma.pizzetti}, we have
\begin{equation}\label{interm.1746456}
(A_{r}-I)[u](x)
-
\big(c_{1}\,{r}^{2}\Delta u(x)
+
c_{2}\,{r}^{4}\Delta^{2}u(x)
+
\cdots
+
c_{m}\,{r}^{2m}\Delta^{m}u(x)
\big)
=
o({r}^{2m})
\end{equation}
as ${r} \to0$ locally uniformly in $\Omega$.
If we define
\[
w(x)=(A_{r}-I)[u](x)
-
\big(
c_{1}\,{r}^{2}\Delta u(x)
+
c_{2}\,{r}^{4}\Delta^{2}u(x)
+
\cdots
+
c_{m}\,{r}^{2m}\Delta^{m}u(x)
\big),
\]
then \eqref{interm.1746456} reads $w(x) = o(r^{2m})$ and  Lemma \ref{epsilon.lemma1} shows that 
\begin{equation}\label{induction.hypothesis.plus2}
    \Big(Q(A_{r})\,(A_{r}-I)^{m-1}\Big)[w](x)=o\big({r}^{2m}\big)
    \end{equation}
as ${r}\to0$ locally uniformly in $\Omega$.  By the definition of $w(x)$ and \eqref{induction.hypothesis.plus2}, we have
\begin{equation*}
\begin{split}
\Big(Q(A_{r})\,(A_{r}-I)^m\Big)[u](x) = &\;c_{1}\, {r}^{2}   \Big(Q(A_{r})\,(A_{r}-I)^{m-1}\Big)[\Delta u](x)\\
&+c_{2}\,{r}^{4}\Big(Q(A_{r})\,(A_{r}-I)^{m-1}\Big)[\Delta^{2}u](x) +
\cdots\\
&+
c_{m}\,{r}^{2m}\Big(Q(A_{r})\,(A_{r}-I)^{m-1}\Big)[\Delta^{m}u](x)+
o({r}^{2m}).
\end{split}
\end{equation*}
Since 
$u$ satisfies \eqref{AMVP.mharmonic.turbo.smooth.propo} by hypothesis, we have
\begin{multline*}
c_{1}{r}^{2}\Big(Q(A_{r})(A_{r}-I)^{m-1}\Big)[\Delta u](x)+\;c_{2}{r}^{4}\Big(Q(A_{r})(A_{r}-I)^{m-1}\Big)[\Delta^{2}u](x)
+\cdots\\
+c_{m}{r}^{2m}\Big(Q(A_{r})(A_{r}-I)^{m-1}\Big)[\Delta^{m}u](x)
=
o({r}^{2m})
\end{multline*}
as ${r}\to0$ locally uniformly in $\Omega$.  By Lemmas \ref{epsilon.lemma2} and \ref{epsilon.lemma1}, we have 
    \[
    \Big(Q(A_{r})(A_{r}-I)^{m-1}\Big)[\Delta^{k} u](x)
=o({r}^{2m-4}),\qquad k=2,3,\ldots,m
    \]
as ${r} \to 0$ locally uniformly in $\Omega$.
Dividing by $c_1r^2$, we have
$$\Big(Q(A_{r})(A_{r}-I)^{m-1}\Big)[\Delta u](x) = o({r}^{2m-2})$$
as ${r}\to0$ locally uniformly in $\Omega$,
which means that $\Delta u$ is $(m-1)$-harmonic by the inductive  assumption \eqref{induction.hypothesis}.  Thus, $u$ is $m$-harmonic as desired and the proof is complete.
\end{proof}

We turn now to proving Theorem \ref{thm.MVP.harmonic.intro.2} under the additional assumption that $u$ is smooth.  This shows that Theorem \ref{thm.mvp.mharmonic.smooth} is optimal, since it is precisely mean-value properties with polynomials of the form $Q(A_r)(A_r-I)^m$ that characterize $m$-harmonic functions. We first prove an auxiliary~result.

\begin{lemma}\label{charro.lebiedzik.raihen}
Given $k=0,1,2,\ldots$, let us define
$u_{k}(x)=x_{1}^{2k}.$ We have 
\begin{equation}\label{genera.lemma.m.Laplacian}
\Delta^{m} u_k(x)=\begin{cases}
\displaystyle
    \frac{(2k)!}{(2k-2m)!}u_{k-m}(x)  & m \le k \\
   \displaystyle 0 & m>k
\end{cases}
\end{equation}
and
\begin{subnumcases}{(A_{r}-I)^j [u_k]=}
   \displaystyle
    (2k)!\,c_{1}^{k}{r}^{2k} & j = k \label{genera.lemma.operator2}\\
    0 & j \textrm{$>$} k.\label{genera.lemma.operator3}
\end{subnumcases}
\end{lemma}

\begin{proof}
The proof of \eqref{genera.lemma.m.Laplacian} is a direct computation.
For the proof of \eqref{genera.lemma.operator2}--\eqref{genera.lemma.operator3}, we proceed by strong induction in $k$.
In view of \eqref{genera.lemma.m.Laplacian}, and given $k=0,1,2,\ldots,$ Pizzetti's formula yields
\begin{equation}\label{pizza.pizzo}
\begin{split}
(A_{r}-I)[u_{k}](x)&=c_{1}{r}^{2}\Delta u_{k}(x)+c_{2}{r}^{4}\Delta^{2} u_{k}(x)+\cdots+c_{k}{r}^{2k}\Delta^{k} u_{k}(x)\\
&=c_{1}{r}^{2}\frac{(2k)!}{(2k-2)!}u_{k-1}(x) +c_{2}{r}^{4}\frac{(2k)!}{(2k-4)!}u_{k-2}(x) +\cdots+c_{k}{r}^{2k}(2k)!\,u_0(x)\\
&=\sum_{l=0}^{k-1}c_{l+1}\,r^{2l+2}\frac{(2k)!}{(2k-2l-2)!}u_{k-l-1}(x).
\end{split}
\end{equation}

In particular, \eqref{pizza.pizzo} yields  $(A_{r}-I)[u_{1}](x)=2c_{1}{r}^{2}$ (which is constant), and, in turn, $(A_{r}-I)^j[u_{1}](x)=0$ for every $j>1$,
proving the base case $k=1$.
Now, let us assume \eqref{genera.lemma.operator2}--\eqref{genera.lemma.operator3} hold true up to a certain $k>1$ (strong induction hypothesis), and prove them for $k+1$.
Let $j\geq k+1$.
From \eqref{pizza.pizzo} and the strong induction hypothesis, we deduce that
\begin{equation}\label{harkonic.mattern}
\begin{split}
&(A_{r}-I)^{j}[u_{k+1}](x)\\
&=c_{1}{r}^{2}\frac{(2k+2)!}{(2k)!}(A_{r}-I)^{j-1}\left[u_{k}\right](x)
+\cdots+c_{k}{r}^{2k}\frac{(2k+2)!}{2!}(A_{r}-I)^{j-1}
\left[u_1\right](x)
\\
&=c_{1}{r}^{2}\frac{(2k+2)!}{(2k)!}(A_{r}-I)^{j-1}\left[u_{k}\right](x).
\end{split}
\end{equation}
 If $j>k+1$, the right-hand side of \eqref{harkonic.mattern} is zero (by the induction hypothesis), and if $j= k+1$,
\[
(A_{r}-I)^{k+1}[u_{k+1}](x)
=c_{1}{r}^{2}\frac{(2k+2)!}{(2k)!}(A_{r}-I)^{k}\left[u_{k}\right](x)
=
(2k+2)!c_{1}^{k+1}{r}^{2k+2},
\]
which completes the proof of \eqref{genera.lemma.operator2}--\eqref{genera.lemma.operator3}.
\end{proof}

As a consequence, we find that the order $o(r^{2m})$ in the mean-value property is optimal. 

\begin{corollary}\label{corollary.optimal.order}
Let $m$ be a positive integer, and let $l\in(0,2m)$. Then, there are smooth functions $u$ that are not  $m$-harmonic ($\Delta^mu\neq0$ in $\Omega$), but satisfy 
$(A_{r}-I)^m[u](x)=o({r}^{l})$ as ${r}\to0$ locally uniformly in $\Omega$.
\end{corollary}

\begin{proof}
Let $m,l$ with $l<2m$ as stated, and let $u_m(x)=x_1^{2m}$. By Lemma \ref{charro.lebiedzik.raihen}, we have
\[
(A_{r}-I)^m [u_m]
=(2m)!\,c_{1}^{m}{r}^{2m}
=
o(r^l)
\]
as ${r}\to0$ locally uniformly in $\Omega$. On the other hand, 
$\Delta^{m} u_m(x)=(2m)!\neq0,$
as desired.
\end{proof}

We are now ready to prove Theorem \ref{thm.MVP.harmonic.intro.2} in the smooth case.

\begin{theorem}\label{thm.mvp.m_harmonic.turbo.boost}
Consider a non-constant real polynomial $P$, and assume the existence of a converse ``asymptotic mean-value property of order $2m$," given by $P$, that characterizes smooth $m$-harmonic functions. That is to say,
 assume that
 every nontrivial  $u\in C^\infty(\Omega)$ that satisfies the ``asymptotic mean-value formula" \begin{equation}\label{AMVP.m_harmonic.gen.P.smooth}
    P(A_{r})[u](x)=o({r}^{2m})  \quad
    \textrm{as}\ {r}\to0~ \textrm{locally uniformly in}\ \Omega,
    \end{equation}
is $m$-harmonic in $\Omega$, but not necessarily  $(m-1)$-harmonic.    
    Then,  $P(A_{r})=Q(A_{r})(A_{r}-I)^m$ for some real polynomial $Q$ with $Q(1)\neq0$.
In particular, the conclusion holds when we replace \eqref{AMVP.m_harmonic.gen.P.smooth} by
    \eqref{MVP.harmonic.gen.P.intro}.
\end{theorem}

\begin{proof}
Let $u\in C^\infty(\Omega)$ be a nontrivial $m$-harmonic function that satisfies \eqref{AMVP.m_harmonic.gen.P.smooth}.
  Notice that there exist  polynomials $Q$ and $R$ such that 
\[
P(A_{r})=Q(A_{r})(A_{r}-I)^{m}+R(A_{r}),
\]
with 
\[
R(A_{r})={a}_{m-1}\,(A_{r}-I)^{m-1}+\cdots+{a}_{1}\,(A_{r}-I)+{a}_0 \,I
\]
for some ${a}_0,\ldots, {a}_{m-1}\in \mathbb{R}$.
Let $x \in \Omega$ and ${r}<\textnormal{dist}(x,\partial\Omega)/\textnormal{deg(P)}$ so that $P(A_{r})[u](x)$ is well-defined. Since $u$ is $m$-harmonic, Proposition \ref{thm.mvp.mharmonic.smooth.second.piece} gives 
\[
(A_{r}-I)^m[u](x)=0.
\]
Then, by \eqref{AMVP.m_harmonic.gen.P.smooth}, we have
\begin{equation}\label{equation.reminder.polyharmonic0}
o({r}^{2m})=P(A_{r})[u](x)
=
Q(A_{r})\big[(A_{r}-I)^{m}[u]\big] (x)+R(A_{r})[u](x)
=R(A_{r})[u](x)
\end{equation}
as ${r}\to0$ locally uniformly in $\Omega$.
Observe that, in principle, the degree of $P$ could  be less than $m$, in which case one would take $Q\equiv0$ and still get \eqref{equation.reminder.polyharmonic0}.

Let us show that $R=0$ and, in turn, that the degree of $P$ is at least $m$.
From \eqref{equation.reminder.polyharmonic0}, we have 
\[
o({r}^{2m})
=
{a}_{m-1}\,(A_{r}-I)^{m-1}[u](x)+\cdots+{a}_{1}\,(A_{r}-I)[u](x)+{a}_0 \,u(x)
\]
as ${r}\to0$ locally uniformly in $\Omega$. Letting $r\to0$, Lemma \ref{epsilon.lemma2} gives ${a}_0=0$ since $u$ is nontrivial. With this information, we can again apply Lemma \ref{epsilon.lemma2}  to obtain ${a}_1=0$ since, otherwise, $u$ would be necessarily harmonic (and hence $(m-1)$-harmonic) by Proposition \ref{thm.mvp.mharmonic.smooth.first.piece}, a contradiction. From here, repeated application of  Lemma \ref{epsilon.lemma2} shows that ${a}_0={a}_1=\ldots={a}_{m-1}=0$, that is, $R=0$.

It only remains to show that $Q(1)\neq0$. Assume to the contrary that
\[
P(A_{r})=\tilde{Q}(A_{r})\,(A_{r}-I)^j
\]
for some $m+1\leq j\leq \textnormal{deg(P)}$ and $\tilde{Q}(1)\neq0$. 
Let us show that, in that case, there are smooth functions satisfying      \eqref{MVP.harmonic.gen.P.intro} (and hence \eqref{AMVP.m_harmonic.gen.P.smooth}) that are not $m$-harmonic, a contradiction with our hypothesis. 
The function $u_m(x)=x_{1}^{2m}$ is one such function. 
For this choice of function, we have $\Delta^{m} u_m = (2m)! \neq0$ and 
$(A_{r}-I)^j[u_m]=0$ for all $j\geq m+1$, see Lemma \ref{charro.lebiedzik.raihen}. Therefore, $P(A_{r})[u_m]=0$ for a function $u_m$ that is not $m$-harmonic, a contradiction with our hypothesis.
\end{proof}


\section{Equivalence of mean-value formulas satisfied by locally integrable functions}\label{section.equivalence}

In this section, we prove  an equivalence between mean-value formulas satisfied by merely locally integrable functions. We show that for $u\in L^1_{\textnormal{loc}}(\Omega)$, any one mean-value property in the family holds if and only if any other does,  where each property in any given pair may independently be exact or asymptotic, and may use either solid or spherical averages.

\begin{theorem}\label{thm.MVP.polyharmonic.equivalence}
Let $\Omega\subset\mathbb{R}^n
$ be an open set, $m$ a positive integer, and $u \in L^1_\textnormal{loc}(\Omega)$. Then, the following are equivalent: 
\begin{enumerate}\itemsep3pt
    \item For  almost every $x\in\Omega$ and  every  ${r}<\textnormal{dist}(x,\partial\Omega)/m$, it holds that
\begin{equation*}
    (A_{r}-I)^m[u](x)=0,
    \end{equation*} 
    or, equivalently,
\begin{equation*}
    u(x)=\sum_{j=1}^{m}\binom{m}{j}(-1)^{j-1}A_{r}^j[u](x).
    \end{equation*}

    \item Given any polynomial  $P$ such that $1$ is a root of multiplicity exactly $m,$ 
\begin{equation*}
    P(A_{r})[u](x)=0  \quad
    \textrm{a.e. $x\in\Omega$ and  every ${r}<\frac{\textnormal{dist}(x,\partial\Omega)}{\textnormal{deg(P)}}$.}
    \end{equation*}    

    \item Given any polynomial  $P$ such that $1$ is a root of multiplicity $m$, 
\begin{equation*}
    P(A_{r})[u](x)=o({r}^{2m})  \quad
    \textrm{as}\ {r}\to0\ \textrm{locally uniformly in}\ \Omega.
    \end{equation*}

\end{enumerate}

Statements (1)--(3) are also equivalent to their counterparts where $A_r$ is replaced by $\mathcal{A}_r$.
 In particular,
   for any two real polynomials  $P$, $Q$ that have 1 as a root of multiplicity exactly $m$,  for every $u\in L^1_\textnormal{loc}(\Omega)$, we have
    \[
    P(A_{r})[u](x)=0 \quad\text{if and only if} \quad Q({\mathcal{A}_{r}})[u](x)=0
    \]
    for almost every $x\in\Omega$ and  every $r<\textnormal{dist}(x,\partial\Omega)/\max\{\textnormal{deg(P)},\textnormal{deg(Q)}\}$.
\end{theorem}

Before proceeding with the proof of Theorem \ref{thm.MVP.polyharmonic.equivalence}, let us prove some preliminary results where we discuss the relationship between mean-value properties and mollification. Let us recall the  standard, compactly supported  mollifier (see \cite{Evans}),
\[
\eta(x)= 
\begin{cases}
        C \exp\left(\frac{1}{|x|^{2}-1}\right) & \textrm{if}\ |x|<1\\
        0 & \textrm{if}\ |x|\ge 1,
\end{cases}
\]
   the constant $C>0$ selected so that $\int_{\mathbb{R}^{n}}\eta\ dx=1$. Then, $\eta \in C^{\infty}(\mathbb{R}^{n})$ and is a radial function.
   Set 
      \begin{equation}\label{eta.delta.eq}
    \eta_{\epsilon}(x)=\frac{1}{{\epsilon}^{n}}\,\eta\left(\frac{x}{\epsilon}\right).
    \end{equation}
   Given $u \in L^1_\textnormal{loc}(\Omega)$, its mollification
   $
   u_{\epsilon}=\eta_{\epsilon}*u
   $
    is well-defined in
   $
   \Omega_{\epsilon}=\{x\in \Omega \hspace{2mm}|\hspace{2mm} \textnormal{dist}(x, \partial \Omega)>{\epsilon}\}.
   $
In the next lemma, we show that averaging and mollification commute.
\begin{lemma}\label{molifier.satisfies.mvp}
    Let $\Omega\subset \mathbb{R}^n$ be an open set, $\epsilon>0,$ and $u \in L^1_\textnormal{loc}(\Omega)$. Then, for every $k=1,2,\ldots$ we have
  \begin{equation}\label{173746238}
  \big(A_{r}^k[u]*\eta_{\epsilon}
  \big)(x)=A_{r}^k[u*\eta_{\epsilon}](x)
  \quad\textrm{for every $x\in\Omega_\epsilon$ and  ${r}<\frac{\textnormal{dist}(x,\partial\Omega_\epsilon)}{k}$}.
  \end{equation}
   The same is true for the spherical average $\mathcal{A}_{r}$.
\end{lemma}
   
  \begin{proof}
   Let us define
   \begin{equation}\label{indicator.chi}
\chi_{r}(x)=\frac{1}{|B_r(0)|}\,\chi_{B_r(0)}(x),
\end{equation}
with $\chi_{B_r(0)}$ the indicator function of the ball $B_r(0)$, and  observe that
      \[
      A_{r}[u](x)=\dashint_{B_{r}(x)}u(y)\,dy=(u*\chi_{r})(x),
      \]
      which is continuous in $x$ by the Lebesgue lemma.
    We proceed by induction. By Fubini's theorem we have
\[
\left(A_{r}[u]*\eta_{\epsilon}\right)(x) 
      =\left((u*\chi_{r})*\eta_{\epsilon}
      \right)(x)=\left((u*\eta_{\epsilon})*\chi_{r}\right)(x) = 
      A_{r}[u*\eta_{\epsilon}](x),
\]
for all $x\in\Omega_\epsilon$ and  ${r}<\textnormal{dist}(x,\partial\Omega_\epsilon)$,
        which proves the base case. 
         Given an integer $k\geq2,$ let us now assume
  \[
  \big(A_{r}^{k-1}[u]*\eta_{\epsilon}
  \big)(x)=A_{r}^{k-1}[u*\eta_{\epsilon}](x)
  \quad
    \textrm{for every $x\in\Omega_\epsilon$ and  ${r}<\frac{\textnormal{dist}(x,\partial\Omega_\epsilon)}{k-1}$},
  \]
        and let us prove \eqref{173746238}. Indeed, given 
         $x\in\Omega_\epsilon$ and  ${r}<\textnormal{dist}(x,\partial\Omega_\epsilon)/k$,   
         Fubini's theorem yields
      \begin{equation*}
      \begin{split}
\left(A_{r}^k[u]*\eta_{\epsilon}\right)(x)
      &=\left(A_{r}\big[A_{r}^{k-1}[u]\big]*\eta_{\epsilon} 
      \right)(x)
      =\left(\big(A_{r}^{k-1}[u]*\chi_{r}\big)*\eta_{\epsilon}
      \right)(x)
      \\
      &=
      \left(\big(A_{r}^{k-1}[u]*\eta_{\epsilon}\big)*\chi_{r}
      \right)(x)
      =
      \left(\big(A_{r}^{k-1}[u*\eta_{\epsilon}]\big)*\chi_{r}
      \right)
      (x)
      \\
      &=
      A_{r}\big[A_{r}^{k-1}[u*\eta_{\epsilon} ]\big]
      (x)
      =
      A_{r}^{k}[u*\eta_{\epsilon}](x),
      \end{split}
      \end{equation*}
    as desired. The same argument holds with the spherical average $\mathcal{A}_{r}$ in place of $A_{r}$, the only difference is the indicator function in \eqref{indicator.chi}, where now we have the corresponding normalized surface measure on $\partial B_r(0)$.  
  \end{proof}

The following lemma shows that  exact mean-value properties are preserved under mollification.

\begin{lemma}\label{polynomial.mvp}
    Let $\Omega\subset \mathbb{R}^n$ be an open set, $u \in L^1_\textnormal{loc}(\Omega)$, and $u_{\epsilon}$ its mollification. Let $P$ be a non-constant, real polynomial. Then we have
    \begin{equation}\label{827.3j.a}
    \textrm{$P(A_{r})[u](x)=0$
    for a.e. $x\in\Omega$ and every ${r}<\frac{\textnormal{dist}(x,\partial\Omega)}{\textnormal{deg}(P)}$,}
    \end{equation}
    if and only if for every open set $U$ such that $\overline{U}\subset\Omega$, and every 
    $0<\epsilon<\textnormal{dist}(U, \partial\Omega)$, it holds that 
    \[
    \textrm{
    $P(A_{r})[u_{\epsilon}](x)=0$
    for all $x\in U$ and every ${r}<\frac{\textnormal{dist}(x,\partial U)}{\textnormal{deg}(P)}$.}
    \]
    The same is true for the spherical average $\mathcal{A}_{r}$.
\end{lemma}

\begin{remark}
    If the polynomial $P$ in Lemma \ref{polynomial.mvp} has no zero-order terms, 
   no term involving the function $u(x)$ is present in the mean-value property, only terms involving the iterated averages $A_r[u](x),A_r^2[u](x),\ldots$, which define continuous functions in $x.$ 
   Then,
    \eqref{827.3j.a} holds pointwise and not just almost everywhere.
\end{remark}

\begin{proof} 
    Assume that $ P(A_{r})[u](x)=0 $
    for a.e. $x\in\Omega$ and  ${r}<\textnormal{dist}(x,\partial\Omega)/\textnormal{deg}(P)$.
    So, we can write 
    \[
        0=P(A_{r})[u](x)= \Big(\sum_{k=0}^{p}{a}_{k}A_{r}^k\Big)[u](x)
    \]
    for some real numbers ${a}_0,\ldots,{a}_p$ with $p=\textnormal{deg}(P)$.
    Fix an open set $U$ such that $\overline{U}\subset\Omega$, and $\epsilon>0$ small enough so that $U\subset\Omega_\epsilon$.
    Then, by Lemma \ref{molifier.satisfies.mvp}, we have  
    \[
    0=\big(P(A_{r})[u]*\eta_{\epsilon}\big)(x) 
    =\sum_{k=0}^{p}{a}_{k}\Big(A_{r}^k[u]*\eta_{\epsilon}\Big)(x)
    =\sum_{k=0}^{p}{a}_{k}\, A_{r}^k[u*\eta_{\epsilon}](x) 
    =P(A_{r})[u_{\epsilon}](x)
    \]
    for all $x\in U$ and  ${r}<\textnormal{dist}(x,\partial U)/\textnormal{deg}(P)$.
    
    Assume now that
      for a fixed open set  $U$ such that $\overline{U}\subset\Omega$
      and all $\epsilon$ small enough
      such that $U \subset\Omega_\epsilon$,  it holds that
    $P(A_{r})[u_{\epsilon}](x)=0$
    for all $x\in U$ and  ${r}<\textnormal{dist}(x,\partial U)/\textnormal{deg}(P)$. 
     Then,
    \[
    0=\lim_{\epsilon \to 0}P(A_{r})[u_{\epsilon}](x)=\lim_{\epsilon \to 0}P(A_{r})[u*\eta_{\epsilon}](x)=\lim_{\epsilon \to 0}\big(P(A_{r})[u]*\eta_{\epsilon}\big)(x)= P(A_{r})[u](x)
    \]
    a.e. in $U$ (and pointwise in $U$ if the polynomial $P$ contains no zero-order terms, since in that case
    $x\mapsto P(A_{r})[u](x)$ is continuous).
    Since $U$ is arbitrary, this completes the proof of the equivalence.
    The same argument holds for the spherical average $\mathcal{A}_{r}$.
\end{proof}

Finally, we show that asymptotic mean-value properties are preserved under mollification.

\begin{lemma}\label{asymptotic.mvp}
    Let $\Omega\subset \mathbb{R}^n$ be an open set,  $u \in L^1_\textnormal{loc}(\Omega)$, and $u_{\epsilon}$ its mollification. Then,
    \begin{equation}\label{58374563}
    P(A_{r})[u](x)=o({r}^{2m})\quad\textrm{as ${r} \to 0$ locally uniformly in $\Omega$}
    \end{equation}
    if and only if for every bounded open set $U$ such that $\overline{U}\subset\Omega$, we~have
    \begin{equation}\label{48574563}
    P(A_{r})[u_{\epsilon}](x)=o({r}^{2m})\quad\textrm{ as ${r} \to 0$ locally uniformly in $U,$}
    \end{equation}
    for every $\epsilon$ small enough with the little-$o$ uniform in $\epsilon$, in the sense  that for every compact set 
$K\subset U$ and $\delta>0$, there exists $r_0>0$ (depending on $U,\delta$ but not on $\epsilon$)
such that
\[
\sup_{x\in K}\big|P(A_{r})[u_{\epsilon}](x)\big|
    \le \delta\, r^{2m}
    \quad \textrm{for all } 0<r<r_0 
    \text{ and all } 0<\epsilon<\min\left\{\textnormal{dist}(U, \partial\Omega),\textnormal{dist}(K, \partial U)\right\}.
\]
The same conclusion holds  with the spherical average $\mathcal{A}_{r}$ in place of the solid average $A_{r}$.

\end{lemma}

\begin{proof}
Let us show the proof for mean-value formulas involving $A_r$. The argument for the spherical average $\mathcal{A}_r$ in place of $A_r$ is identical. 
Assume first that \eqref{58374563} holds and let us prove \eqref{48574563}.
 Let $U$ be a bounded open set with $\overline{U}\subset \Omega$, $K\subset U$ be a compact set, and let $\delta>0$. By \eqref{58374563}, there exists $r_0>0$ such that
\begin{equation}\label{eq:u-on-Ubar}
\operatorname*{ess\,sup}_{x\in \overline{U}}
\big|P(A_r)[u](x)\big|
\le \delta\,r^{2m}
\quad\text{for all }0<r<r_0.
\end{equation}
Fix $0<r<r_0$ and take $\epsilon>0$ such that
$0<\epsilon<\min\{\textnormal{dist}(U,\partial\Omega),\textnormal{dist}(K,\partial U)\}.$
Then for each $x\in K$ we have $B_\epsilon(x)\subset U\subset \Omega_\epsilon$, and 
we can apply Lemma~\ref{molifier.satisfies.mvp} as in the proof of Lemma~\ref{polynomial.mvp} to get 
\[
|P(A_r)[u_\epsilon](x)|
= \big|(P(A_r)[u] * \eta_\epsilon)(x)\big|
\le \int_{U} \big|P(A_r)[u](y)\big|\;\eta_\epsilon(x-y)\,dy 
\le
 \operatorname*{ess\,sup}_{y\in \overline{U}} \big|P(A_r)[u](y)\big|.
\]
Taking the supremum over $x\in K$ and using \eqref{eq:u-on-Ubar} gives
\[
\sup_{x\in K}|P(A_r)[u_\epsilon](x)|
\le \delta\,r^{2m}
\]
for all $0<r<r_0$
and all $0<\epsilon<\min\{\textnormal{dist}(U,\partial\Omega),\textnormal{dist}(K,\partial U)\}$,
as desired.

Assume now that \eqref{48574563} holds and let us show \eqref{58374563}.
Let $K\subset\Omega$ be compact and fix $\delta>0$. Choose a bounded open set $U$ such that
$K\subset U$ and $\overline{U}\subset\Omega.$
By \eqref{48574563}, there exists $r_0>0$ (depending on $U,\delta$ but not on $\epsilon$)
such that
\[
\sup_{x\in K}\big|P(A_{r})[u_{\epsilon}](x)\big|
    \le \delta\, r^{2m}
\]
 for all $0<r<r_0$  and all  $0<\epsilon<\min\left\{\textnormal{dist}(U, \partial\Omega),\textnormal{dist}(K, \partial U)\right\}$.
Since $U\subset\Omega_\epsilon$, we can invoke Lemma~\ref{molifier.satisfies.mvp} once more to get
\begin{equation}\label{u.epsilon.bound}
\sup_{x\in K} \big|(P(A_r)[u] * \eta_\epsilon)(x)\big|
=
\sup_{x\in K}\big|P(A_{r})[u_{\epsilon}](x)\big|
    \le \delta\, r^{2m}
\end{equation}
under the same conditions on $r$ and $\epsilon$.
Now, we want to show that
\[
\operatorname*{ess\,sup}_{x\in K} \big|P(A_r)[u](x)\big| \le \delta\,r^{2m},
\]
which is \eqref{58374563}. Suppose, for the sake of contradiction, that for some fixed $r\in(0,r_0)$ we have
\[
\operatorname*{ess\,sup}_{x\in K} \big|P(A_r)[u](x)\big|
> \delta\,r^{2m},
\]
and choose $\tau>0$ such that
\[
\operatorname*{ess\,sup}_{x\in K} \big|P(A_r)[u](x)\big| \ge \delta\,r^{2m} + 3\tau.
\]
Define
\[
E_\tau = \big\{x\in K : \big|P(A_r)[u](x)\big| \ge \delta\,r^{2m} + 2\tau\big\}.
\]
Then, by the definition of essential supremum, $|E_\tau|>0$.  On the other hand, for a.e. $x\in E_\tau$
\[
(P(A_r)[u] * \eta_\epsilon)(x) \to P(A_r)[u](x)
\quad\text{as }\epsilon\to 0
\]
by the properties of mollification.
 Hence, for any such $x\in E_\tau$, there is $\epsilon_x>0$ such that 
\[
\big|(P(A_r)[u] * \eta_\epsilon)(x) - P(A_r)[u](x)\big| < \tau
\]
for all~$\epsilon\in(0,\epsilon_x)$. We conclude that
\[
\big|(P(A_r)[u] * \eta_\epsilon)(x)\big|
\ge \big|P(A_r)[u](x)\big| - \tau
\ge \delta\,r^{2m} + \tau
> \delta\,r^{2m},
\]
a contradiction with \eqref{u.epsilon.bound}.
\end{proof}

We are now ready for the
\begin{proof}[Proof of Theorem \ref{thm.MVP.polyharmonic.equivalence}]
We begin by proving the equivalence between statements (1) and (2). In fact, let us prove a more general equivalence. Let $P_1,P_2$ be two real polynomials, not necessarily of the same degree, that have $1$ as a root of multiplicity exactly $m.$ Given $u \in L^1_\textnormal{loc}(\Omega)$ such that
\[
  P_1(A_{r})[u](x)=0  \quad
    \textrm{for a.e. $x\in\Omega$ and every ${r}<\frac{\textnormal{dist}(x,\partial\Omega)}{\textnormal{deg($P_1$)}}$,}
\]
we are going to show that the same statement holds with $P_2$ in place of $P_1.$

Fix an open set $U$ such that $\overline{U}\subset\Omega$. By Lemma~\ref{polynomial.mvp}, for every $0<\epsilon<\textnormal{dist}(U, \partial\Omega)$, we have
\begin{equation}\label{p1.theorem.proof.290}
    P_1(A_{r})[u_{\epsilon}](x)=0\ 
    \textrm{for all $x\in U$ and every}\ r<\frac{\textnormal{dist}(x,\partial U)}{\textnormal{deg}(P)}.
\end{equation}
This allows us to work with smooth functions. Then, by 
Theorems \ref{thm.mvp.m_harmonic.turbo.boost} and \ref{thm.mvp.mharmonic.smooth}, $u_\epsilon$ also satisfies 
\[
    P_2(A_{r})[u_\epsilon](x)=0  \quad
    \textrm{for every $x\in U$ and  ${r}<\frac{\textnormal{dist}(x,\partial U)}{\textnormal{deg($P_2$)}}$}
\]
for every $\epsilon>0$ in the same range as before.
Since the set $U$ is arbitrary, Lemma \ref{polynomial.mvp} implies that  $P_2(A_r)[u](x)=0$ for almost every $x$ in $\Omega$ and every  $r<{\text{dist}(x,\partial\Omega)/\text{deg}(P_2
)},$ as desired.
The same argument with the roles of $P_1$ and $P_2$ reversed proves a full equivalence and, in particular, that statements (1) and~(2) are equivalent.

We will now show that statements (2) and (3) are equivalent. 
 As before, let $P_1,P_2$ be two real polynomials, not necessarily of the same degree, that have $1$ as a root of multiplicity exactly $m.$ 
First, assume statement (3) holds for $P_1$ and let us prove that statement (2) holds for $P_2$. More precisely, assume that $u \in L^1_\textnormal{loc}(\Omega)$ satisfies
\begin{equation}\label{in.the.proof.of.them.4.1.statement.1.1}
    P_1(A_{r})[u](x)=o({r}^{2m})  \quad
    \textrm{as}\ {r}\to0\ \textrm{locally uniformly in}\ \Omega.
\end{equation}
Let $u_{\epsilon}$ be the mollification of $u$, and consider a fixed bounded open set $U$ such that $\overline{U}\subset\Omega$.
By Lemma \ref{asymptotic.mvp}, it holds 
\begin{equation}\label{jhfher8.383}
    P_1(A_{r})[u_{\epsilon}](x)=o({r}^{2m})\quad\textrm{ as ${r} \to 0$ locally uniformly in $U,$}
\end{equation}
with the little-$o$ uniform in $\epsilon$ in the sense of Lemma \ref{asymptotic.mvp}. In particular, $0<\epsilon<\textnormal{dist}(U, \partial\Omega)$.
Since $u_\epsilon$ is smooth,  
Theorems~\ref{thm.mvp.m_harmonic.turbo.boost} and \ref{thm.mvp.mharmonic.smooth} apply, and we get
\begin{equation}\label{jdhfyf736hr7}
P_2(A_{r})[u_\epsilon](x)=0  \quad
    \textrm{for every $x\in U$ and  ${r}<\frac{\textnormal{dist}(x,\partial U)}{\textnormal{deg($P_2$)}}$}
\end{equation}
for all $0<\epsilon<\textnormal{dist}(U, \partial\Omega)$. Since $U$ is arbitrary, by 
Lemma \ref{polynomial.mvp}, the mean-value property \eqref{jdhfyf736hr7} transfers to $u$, i.e.,
\begin{equation}\label{in.the.proof.of.them.4.1.statement.2.1}
   P_2(A_{r})[u](x)=0  \quad
    \textrm{for a.e. $x\in\Omega$ and every  ${r}<\frac{\textnormal{dist}(x,\partial\Omega)}{\textnormal{deg($P_2$)}}$},
\end{equation}
that is, statement (2) holds for $P_2$ as desired.

Let us now assume \eqref{in.the.proof.of.them.4.1.statement.2.1} and prove \eqref{in.the.proof.of.them.4.1.statement.1.1}.
Fix a bounded open set $U$ such that $\overline{U}\subset\Omega$ and $0<\epsilon<\textnormal{dist}(U, \partial\Omega)$. The argument proceeds in three steps.
First, using Lemma \ref{polynomial.mvp} as before, one can show that the mean-value property 
\eqref{in.the.proof.of.them.4.1.statement.2.1}
transfers to 
$u_{\epsilon}$ pointwise in $U$ and, by Theorem \ref{thm.mvp.mharmonic.smooth}, that 
$u_{\epsilon}$ also satisfies \eqref{p1.theorem.proof.290} pointwise in $U$. 
Next, the key observation is that
\eqref{p1.theorem.proof.290} also means $u_\epsilon$ immediately satisfies \eqref{jhfher8.383} with the little-$o$ uniform in $\epsilon$ in the sense of Lemma \ref{asymptotic.mvp}: Given a compact set 
$K\subset U$, $\delta>0$, and $0<\epsilon<\min\left\{\textnormal{dist}(U, \partial\Omega),\textnormal{dist}(K, \partial U)\right\}$,
\[
\sup_{x\in K}\big|P_1(A_{r})[u_{\epsilon}](x)\big|
=0
    \le \delta\, r^{2m}
    \quad \textrm{for all } 0<r<r_0=\frac{\textnormal{dist}(K,\partial U)}{\textnormal{deg($P_1$)}}
\]
(note that $r_0$ is independent of $\epsilon$).
Finally, since $U\subset\Omega$ is arbitrary, we can apply Lemma  \ref{asymptotic.mvp} one final time to transfer the asymptotic mean-value property to $u$ in $\Omega$ and obtain \eqref{in.the.proof.of.them.4.1.statement.1.1} as desired.

The proofs of the equivalences among relations involving $\mathcal{A}_r$ and $A_r$ are identical once we notice that mean-value properties involving $\mathcal{A}_r$ can also be related to those involving $A_r$ using Theorems~\ref{thm.mvp.m_harmonic.turbo.boost} and \ref{thm.mvp.mharmonic.smooth}.
\end{proof}


\section{Locally integrable functions that satisfy a mean-value property are smooth}\label{section.regularity}

In this section we prove Theorem \ref{thm.regularity}, which follows from the following result.

\begin{theorem}\label{smooth function.proof}
    Let $P$ be a real polynomial  which has 1 as a root with multiplicity $m\geq1$. Assume that $u \in L^1_\textnormal{loc}(\Omega)$ satisfies either
    \begin{equation}\label{smoothness.MVP}
            P(A_{r})[u](x)=o({r}^{2m})  \quad
    \textrm{as}\ {r}\to0\ \textrm{locally uniformly in}\ \Omega,
    \end{equation}
or \eqref{smoothness.MVP} with $\mathcal{A}_r$ in place of $A_r$. Then, $u\in C^\infty(\Omega).$  
In particular, the conclusion holds when we replace \eqref{smoothness.MVP} by
\[
    P(A_{r})[u](x)=0  \quad
    \textrm{for a.e. $x\in\Omega$ and every ${r}<\frac{\textnormal{dist}(x,\partial\Omega)}{\textnormal{deg(P)}}$,
    }
\]
or the corresponding statement for $\mathcal{A}_r$.
\end{theorem}

The proof expands on a well-known idea for the classical  mean-value property for harmonic functions, which  can be found in \cite[Theorem 6, p. 28]{Evans}.

\begin{proof}
    Assume that $u \in L^1_\textnormal{loc}(\Omega)$ and satisfies \eqref{smoothness.MVP}. 
    Then, by Theorem \ref{thm.MVP.polyharmonic.equivalence}, we have proved that $u$ also satisfies $(\mathcal{A}_{r}-I)^{m}[u](x)=0$ a.e. in $\Omega$ for every sufficiently small $r$.

Fix $\epsilon>0$, and consider $\Omega_{\epsilon}=\{x\in \Omega: \textnormal{dist}(x,\partial\Omega)\ge \epsilon\}$.
    Let us define   \begin{equation*}
        v(x)=\sum_{k=0}^{m-1}(\mathcal{A}_{r}-I)^{k}[u](x),
    \end{equation*}
    which is well-defined for all $x\in\Omega_\epsilon$ and $r<\epsilon/(m-1)$.
    Then, by the binomial theorem and the hockey-stick identity for binomial coefficients, we can write 
    \[
        v(x)=\sum_{k=0}^{m-1}\sum_{j=0}^{k}\binom{k}{j}(-1)^{j}\mathcal{A}_{r}^{j}[u](x)=\sum_{j=0}^{m-1} \binom{m}{j+1}(-1)^{j}\mathcal{A}_{r}^{j}[u](x).
    \]
    The function $v$ thus defined has the  key property that its spherical average equals $u$ almost everywhere, i.e., \begin{equation}\label{average.v.eq}
        \begin{split}
       \mathcal{A}_{r}[v](x)&=\mathcal{A}_{r}\bigg[\sum_{j=0}^{m-1} \binom{m}{j+1}(-1)^{j}\mathcal{A}_{r}^{j}[u](x)\bigg]\\&=\sum_{j=0}^{m-1} \binom{m}{j+1}(-1)^{j}\mathcal{A}_{r}^{j+1}[u](x)= \sum_{k=1}^{m} \binom{m}{k}(-1)^{k-1}\mathcal{A}_{r}^{k}[u](x)=u(x),
       \end{split}
    \end{equation}
upon possibly  reducing the range of $r$ so that $r<\epsilon/m$ and $\mathcal{A}_{r}[v]$ is  well-defined in $\Omega_\epsilon$ for all such~$r$.

    We will now prove that $u$ is smooth by demonstrating that it coincides with the mollification of~$v$, that is, $u \equiv v_{\epsilon}$ in $\Omega_{\epsilon}$.
    Let $\eta_{\epsilon}$ be the standard mollifier, defined in \eqref{eta.delta.eq}. Then, if $x\in \Omega_{\epsilon}$, a change to polar coordinates and expression \eqref{average.v.eq} yield
\begin{equation*}
        \begin{split}
          v_{\epsilon}(x)=(v*\eta_{\epsilon})(x)&=\int_{\mathbb{R}^{n}}v(y)\eta_{\epsilon}(x-y)\, dy=\frac{1}{\epsilon^{n}}\int_{B_{\epsilon}(x)}v(y)\eta\bigg(\frac{|x-y|}{\epsilon}\bigg)\, dy\\&=
          \frac{1}{\epsilon^{n}}\int_{0}^{\epsilon}\int_{\partial B_{r}(x)}v(y)\, d\mathcal{H}_{y}^{n-1}\, \eta\bigg(\frac{r}{\epsilon}\bigg)\, dr =
          \frac{1}{\epsilon^{n}}\int_{0}^{\epsilon}|\partial B_{r}(x)|\,\mathcal{A}_{r}[v](x)\, \eta\bigg(\frac{r}{\epsilon}\bigg)\, dr\\&
          =\frac{u(x)}{\epsilon^{n}}\int_{0}^{\epsilon}\int_{\partial B_{r}(0)}\eta\bigg(\frac{r}{\epsilon}\bigg)\,d\mathcal{H}_{y}^{n-1}\, dr=u(x)\int_{B_{\epsilon}(0)}\eta_{\epsilon}(y)\, dy=u(x).
        \end{split}
    \end{equation*}
    We have proved that $u \equiv v_{\epsilon}$
    in $\Omega_{\epsilon}$. Since 
    $v_{\epsilon}\in C^{\infty}(\Omega_{\epsilon})$ and $\epsilon$ is arbitrary, the proof is complete.
\end{proof}


\section{Proof of Theorems \ref{thm.MVP.polyharmonic.intro} and \ref{thm.MVP.harmonic.intro.2} for locally integrable functions}\label{section.general.case}

Let us prove Theorem \ref{thm.MVP.polyharmonic.intro}. We first observe that, by Theorem \ref{thm.MVP.polyharmonic.equivalence}, the mean-value properties in  statements (2), (3) of Theorem \ref{thm.MVP.polyharmonic.intro} and (4)--(6) of Corollary \ref{remark.with.extra.equivalences} are all equivalent. 
Therefore, it is enough to prove the equivalence of  statements (1) and (4).

Let us assume statement (1) holds, that is, assume  $u$ be $m$-harmonic in $\Omega$, in the sense that  $u\in C^{2m}(\Omega)$ and $\Delta^mu=0$ in $\Omega$ in the classical sense. Our goal is to show that statement (4) holds, or, equivalently, that $u$ satisfies \eqref{MVP.mharmonic.main.intro}.

Given $\epsilon>0$, let $\Omega_{\epsilon}=\{x\in \Omega: \textnormal{dist}(x,\partial\Omega)\ge \epsilon\}$ and let $\eta_{\epsilon}$ be the standard mollifier. 
Let $u_{\epsilon}=u*\eta_{\epsilon}$ be the mollification of $u$, well defined in $\Omega_{\epsilon}$. 
By the properties of the mollification, $u_\epsilon\in C^{\infty}(\Omega_{\epsilon})$ and  is $m$-harmonic in $\Omega_\epsilon$; hence, by Theorem \ref{thm.mvp.mharmonic.smooth}, $u_\epsilon$ satisfies \eqref{MVP.mharmonic.main.intro} for  every $x\in\Omega_\epsilon$.

From Lemma \ref{molifier.satisfies.mvp}, for every positive integer $k$ we have
$A_{r}^k[u_{\epsilon}](x)=\big(A_{r}^k[u]*\eta_{\epsilon}
  \big)(x)$ in $\Omega_\epsilon$ for  ${r}<\textnormal{dist}(x,\partial\Omega_\epsilon)/k$. Since $x\mapsto A_{r}^k[u](x)$ is  continuous by the Lebesgue lemma, we have $$A_{r}^k[u_{\epsilon}]=A_{r}^k[u]*\eta_{\epsilon}\to A_{r}^k[u]\quad\textrm{as}\ \epsilon\to0$$
  uniformly on compact subsets of $\Omega$. Furthermore, $u_\epsilon\to u$ a.e. in $\Omega$ as $\epsilon\to0,$ and we deduce that $u$ satisfies
\eqref{MVP.mharmonic.main.intro} a.e. in $\Omega$, as desired.

To prove the converse, i.e., that statement (4) implies statement (1),
notice that
$u\in C^{\infty}$ by Theorem \ref{smooth function.proof}. Therefore, assuming any one of the mean-value properties in  statements (2)--(6) (which, again, are equivalent), statement (1) also holds  by Theorem \ref{thm.mvp.mharmonic.smooth}, completing the proof of Theorem \ref{thm.MVP.polyharmonic.intro}. 

Now, to prove Theorem \ref{thm.MVP.harmonic.intro.2}, first we observe that if $u$ satisfies the mean-value property in \eqref{AMVP.m_harmonic.gen.P.intro}, then 
 Theorem \ref{smooth function.proof} implies that $u\in C^{\infty}(\Omega)$. 
 Thus, Theorem \ref{thm.mvp.m_harmonic.turbo.boost} applies, and the proof of Theorem~\ref{thm.MVP.harmonic.intro.2} is complete.

\appendix\section{Pizzetti's expansions}\label{appendix.pizzetti.formulas}
For convenience, the following result states the Pizzetti formula as used above. The proof of statement (1) below closely follows \cite{O}. It is included for completeness, particularly to establish the local uniformity of the remainder (in the sense of Definition \ref{loc_unif}) in statement (2), which is not addressed in \cite{O}. The proof of statement (3)  follows \cite{Lysik.2011}.

\begin{lemma}\label{appendix.lemma.pizzetti}
Consider an open set $\Omega \subset \mathbb{R}^n$ and a function $u \in C^{2m}(\Omega).$
Define the constants
\[
c_{k}=\frac{1}{2^{k}k!\prod_{j=1}^{k}(n+2j)}
\qquad\textrm{and}\qquad
d_k=\frac{n+2k}{n}\,c_{k}=\frac{1}{2^{k}k!\prod_{j=0}^{k-1}(n+2j)}.
\]
It holds that 
\begin{enumerate}
    \item The mean-value of $u$ over the ball $B_{r}(x)$ and the sphere $\partial B_{r}(x)$  can be represented    as $r \to 0$~by the Pizzetti formulas, 
     \begin{equation} \label{2.series.expansion.mvp.ball}
        A_{r}[u](x)=\dashint_{B_{r}(x)}u(y)\,dy=u(x)+\sum_{k=1}^{m}c_k\,\Delta^{k}u(x)\,r^{2k}+o(r^{2m}),
    \end{equation}
and 
    \begin{equation} \label{2.series.expansion.mvp.surface}
        \mathcal{A}_{r}[u](x)=\dashint_{\partial B_{r}(x)}u(y)\,dy=u(x)+\sum_{k=1}^{m}d_k\,\Delta^{k}u(x)\,r^{2k}+o(r^{2m}).
    \end{equation}

    \item If $u\in C^{2m+1}(\Omega)$, then the remainder in formulas \eqref{2.series.expansion.mvp.ball}, \eqref{2.series.expansion.mvp.surface} is locally uniform in the sense of Definition \ref{loc_unif}.

    \item In the special case that $u$ is $m$-harmonic, the Pizzetti formulas \eqref{2.series.expansion.mvp.ball}, \eqref{2.series.expansion.mvp.surface} are exact, i.e.,
    \begin{equation}\label{pizzetti.exact}
    A_{r}[u](x)=u(x)+\sum_{k=1}^{m-1}c_{k}\,\Delta^{k}u(x)\,r^{2k}
    \quad\textrm{and}\quad
    \mathcal{A}_{r}[u](x)=u(x)+\sum_{k=1}^{m-1}d_{k}\,\Delta^{k}u(x)\,r^{2k}
    \end{equation}
    for  every $x\in\Omega$ and    ${r}<\textnormal{dist}(x,\partial\Omega)$ (when $m=1$, both sums in \eqref{pizzetti.exact} are zero).
\end{enumerate}
\end{lemma}

\begin{proof}
We prove the results for the solid average $A_r$,
since the proofs for the spherical average $\mathcal{A}_r$ are essentially identical.

\medskip
\noindent1.\quad{}Since $u \in C^{2m}(\Omega),$ $D^{\alpha}u(x)$ is defined for all multi-indices $|\alpha| \le 2m$ and the $2m$-th order Taylor polynomial of $u$ about $x$ is given by
    \begin{equation*}
        T_{2m,x}(y)=\sum_{|\alpha|\le 2m}\frac{D^{\alpha}u(x)}{\alpha!}(y-x)^{\alpha}.
    \end{equation*}
As usual, for a multi-index $\alpha=(\alpha_1, \cdots, \alpha_n)$, we denote
$|\alpha| = \alpha_1+\cdots+\alpha_n$, 
$\alpha!=\alpha_1!\cdots\alpha_n!,$ and
    \begin{equation*}
        D^{\alpha}u(x)=\frac{\partial^{|\alpha|}u}{\partial x_1^{\alpha_1}\cdots\partial x_n^{\alpha_n}}(x), \qquad (y-x)^{\alpha}=(y_1-x_1)^{\alpha_1}\cdots(y_n-x_n)^{\alpha_n}.
    \end{equation*}
Then, for $y \in B_r(x)$, we have the Taylor expansion about $x$, given by
    \begin{equation}\label{2.taylor.poly}
        u(y)=T_{2m,x}(y)+o\big(|y-x|^{2m}\big) \quad \textrm{as}\quad |y-x|\to 0.
    \end{equation}
We intend to integrate \eqref{2.taylor.poly} term by term over a ball $B_r(x)$. First observe that
    \begin{equation}\label{2.taylor.expansion}
     \int_{B_{r}(x)} T_{2m,x}(y)\, dy
     =
     u(x)+
     \sum_{k=1}^{m}\frac{|B_{r}(x)|}{2^{k}k!\prod_{j=1}^{k}(n+2j)}\,\Delta^{k}u(x)\,r^{2k},
    \end{equation}
which follows from the following calculus identity (see \cite{O, Folland})
\begin{equation}\label{lemma.monomials}
\frac{1}{\alpha!} \int_{B_{r}(x)}(y-x)^{\alpha}\, dy 
=
\left\{
\begin{aligned}
&\frac{|B_{r}(x)|}{2^{k}\mu!\prod_{j=1}^{k}(n+2j)}\,r^{2k}&&
\textrm{if $\alpha=2\mu$ and $|\alpha|=2k$},\\
&0 &&\textrm{otherwise,}
\end{aligned}
\right.
\end{equation}
and the multinomial expansion of $\Delta^k$, i.e.,
\begin{equation*}
\Delta^{k}u(x)=\left(\frac{\partial^2}{\partial x_1^2}+\cdots+\frac{\partial^2}{\partial x_n^2}\right)^k u(x)=\sum_{\substack{|\alpha|=2k\\\alpha=2\mu}}\frac{k!}{\mu!}\,D^{\alpha}u(x).
\end{equation*}
Therefore, from \eqref{2.taylor.poly}, \eqref{2.taylor.expansion}, we get
\[
\int_{B_{r}(x)} u(y)\, dy= \int_{B_{r}(x)} T_{2m,x}(y)\, dy + |B_{r}(x)|\,o(r^{2m})\quad\textrm{as}\quad r\to0,
\]
and dividing by $|B_{r}(x)|$, gives \eqref{2.series.expansion.mvp.ball}.

\medskip
\noindent2.\quad{}Assume now $u\in C^{2m+1}(\Omega)$ and let us show that the little-$o$ in \eqref{2.series.expansion.mvp.ball} is locally uniform as in Definition \ref{loc_unif}.
Let $K$ be a compact subset of $\Omega$, $\epsilon>0$ fixed, and two points $x,y\in K$.
 Since $u\in C^{2m+1}(\Omega)$, 
 we can rewrite \eqref{2.taylor.poly} with 
the remainder in  Lagrange's form, i.e.,
\begin{equation}\label{taylorwithlagrangeremainder}
u(y)=T_{2m,x}(y)+R_{2m,x}(y).
\end{equation}
with
\[
R_{2m,x}(y)= \sum_{|\alpha|= 2m+1} \frac{D^{\alpha}u(\xi_y)}{\alpha!} (y-x)^{\alpha},
\]
and $\xi_y\in K$ in the line segment between $y$ and $x$.
Integrating \eqref{taylorwithlagrangeremainder} over a ball $B_r(x)\subset K$ and using \eqref{2.taylor.expansion}, we obtain
\begin{equation}\label{identity.with.remainder}
\dashint_{B_{r}(x)} u(y)\, dy
-
u(x)
-
\sum_{k=1}^{m}\frac{1}{2^{k}k!\prod_{j=1}^{k}(n+2j)}\,\Delta^{k}u(x)\,r^{2k}
=
\dashint_{B_{r}(x)} R_{2m,x}(y)\, dy.
\end{equation}
Let us now  estimate the right-hand side of \eqref{identity.with.remainder}. Observe that  \eqref{lemma.monomials} does not apply directly since, contrary to before, we cannot factor $D^{\alpha}u(\xi_y)$ out of the integral due to the dependence of $\xi_y$ on $y$.

For all the multi-indices $\alpha$ with $|\alpha|=2m+1,$ the derivatives $D^\alpha u$ are continuous over the compact set $K$. Therefore, there exists a constant $M>0$ (depending on $K,m,n$ but not $x,\xi_y$), such that
\begin{equation}\label{estimete.remainder}
    \begin{split}
        \sup_{y\in B_r(x)}|R_{2m,x}(y)|\le \sup_{y\in  B_r(x)} \left(\sum_{|\alpha| = 2m+1} \frac{|D^{\alpha}u(\xi_y)|}{\alpha!}|y-x|^{2m+1}\right) \le M\,r^{2m+1}.
    \end{split}
\end{equation}
Then, from \eqref{identity.with.remainder} and \eqref{estimete.remainder}, we get
\[
\sup_{x\in K}\left|\,
\dashint_{B_{r}(x)} u(y)\, dy
-
u(x)-
\sum_{k=1}^{m}\frac{1}{2^{k}k!\prod_{j=1}^{k}(n+2j)}\,\Delta^{k}u(x)\,r^{2k}
\right|
\le M\,r^{2m+1}\leq\epsilon\,r^{2m}
\]
for all $0<r<r_0=\epsilon/M,$
which completes the proof.

\medskip
\noindent3.\quad{}We prove the statement by induction on $m$ as in \cite{Lysik.2011}. Let us first prove the formula for the solid average. We intend to prove that for every integer $m\geq 1,$ it holds that
 \begin{equation}\label{pizzetti.exact.proof}
 \begin{split}
 \textrm{For every $v\in C^{2m}(\Omega)$ such that $\Delta^mv=0$,}&\\
    A_{r}[v](x)=v(x)+\sum_{k=1}^{m-1}c_{k}\,\Delta^{k}v(x)\,r^{2k}&
\quad\textrm{for all $x\in\Omega$ and    ${r}<\textnormal{dist}(x,\partial\Omega)$.}
\end{split}
\end{equation}

If $m=1$ any $u\in C^{2}(\Omega)$ such that $\Delta u=0$ satisfies the classical mean value property, giving
$A_r[u](x)=u(x)$, as claimed.
Assume now that \eqref{pizzetti.exact.proof} holds for some integer $m> 1$, and  let us show it then holds for $m+1$. In particular, let $u\in C^{2m+2}(\Omega)$ satisfy $\Delta^{m+1}u=0$, so that $v=\Delta u\in C^{2m}(\Omega)$ and $\Delta^{m}v=\Delta^{m+1}u=0$. By the inductive hypothesis we have 
\begin{equation}\label{Appendix.3.first}
A_{r}[\Delta u](x)=\sum_{k=0}^{m-1}c_{k}\,\Delta^{k+1}u(x)\,r^{2k}
\end{equation}
for every admissible $r$.
On the other hand, a direct calculation (see \cite{Lysik.2011} for details) shows that
\begin{equation}\label{Appendix.3.second}
\frac{n}{r}\,\frac{\partial}{\partial r}\left(\Big(\frac{r}{n}\frac{\partial}{\partial r}+1\Big)\,A_r[u](x)\right)
= A_r[\Delta u](x).
\end{equation}

From \eqref{Appendix.3.first} and \eqref{Appendix.3.second}, we get the following non-homogeneous ODE satisfied by $A_r[u](x)$:
\[
\frac{n}{r}\,\frac{\partial}{\partial r}\left(\Big(\frac{r}{n}\frac{\partial}{\partial r}+1\Big)\,A_r[u](x)\right)
=\sum_{k=0}^{m-1}c_{k}\,\Delta^{k+1}u(x)\,r^{2k}.
\]
Integrating, we have
\begin{equation}\label{ODE.to.solve}
\Big(\frac{r}{n}\frac{\partial}{\partial r}+1\Big)\,A_r[u](x)
=\sum_{k=0}^{m-1}\frac{c_{k}}{n(2k+2)}\,\Delta^{k+1}u(x)\,r^{2k+2}+c
\end{equation}
for some integration constant $c$. The general solution of the homogeneous equation is $Cr
^{-n}$ with $C$ a constant. A particular solution of 
\[
\Big(\frac{r}{n}\frac{\partial}{\partial r}+1\Big)\,A_r[u](x)
=\frac{c_{k}}{n(2k+2)}\,\Delta^{k+1}u(x)\,r^{2k+2}
\]
is $c_{k+1}\,\Delta^{k+1}u(x)\,r^{2k+2}$. 
Therefore, the general solution of \eqref{ODE.to.solve} is
\[
A_r[u](x) = C r^{-n} + 
\sum_{k=0}^{m-1}c_{k+1}\,\Delta^{k+1} u(x)\, r^{2k+2} + c,
\]
for constants $C,c$ yet to be determined. Since 
\[
\lim_{r\to 0}A_r[u](x)=u(x)
\qquad
\textrm{and}
\qquad
\lim_{r\to 0}r^nA_r[u](x)=0,
\]
we have $C=0$ and $c= u(x),$ which completes the inductive step.
Finally, following \cite{Lysik.2011}, we obtain the formula for the spherical average by using that 
\[
\Big(\frac{r}{n}\frac{\partial}{\partial r}+1\Big)\,A_r[u](x)
=\mathcal{A}_r[u](x).\qedhere
\]
\end{proof}

%


\begin{thebibliography}{BH}

\bibitem{Aronszajn.et.al.1983}N. Aronszajn, T.M. Creese, L.J. Lipkin; \emph{Polyharmonic Functions}, Clarendon Press, Oxford (1983).

\bibitem{BCMR}
P. Blanc, F. Charro, J. J. Manfredi, J. D. Rossi; \emph{Asymptotic mean-value formulas for solutions of general second-order elliptic equations,} Advances in Nonlinear Studies, vol. 22 (2022), no. 1, pp. 118-142. 

\bibitem{Blaschke} W. Blaschke; \emph{Ein mittelwertsatz und eine kennzeichnende eigenschaft des logarithmischen potentials}, Ber. Verh. S\"achs. Akad. Wiss., Leipziger, 68, (1916), 3--7.

\bibitem{BHP} J. H. Bramble,  L. E. Payne; \emph{Mean-value theorems for polyharmonic functions}, The American Mathematical Monthly 73, no. 4P2 (1966), 124-127.

\bibitem{Caramanica.1987}A. Caramanica; \emph{Teoremi di media e formule di maggiorazione per funzioni poliarmoniche,} Universit\`a degli Studi di Roma ``La Sapienza," 1987.

\bibitem{Caramuta.Cialdea.2014} P. Caramuta, A. Cialdea; \emph{mean-value theorems for polyharmonic functions: A conjecture by Picone}, Analysis 34, no. 1 (2014), 51-66.

\bibitem{CM} M.T. Cheng; \emph{On a theorem of Nicolesco and generalized Laplace operators}, Proceedings of the American Mathematical Society 2, no. 1 (1951), 77-86.
  
  \bibitem{Chui.2009}
 C. K. Chui; \emph{An MRA approach to surface completion and image inpainting}, Applied and Computational Harmonic Analysis , vol. 26  (2009), no. 2, pp. 270--276.


\bibitem{Evans}
L. C. Evans; \emph{Partial differential equations}, Vol. 19. American Mathematical Soc., 2010.

\bibitem{Fichera} G. Fichera; \emph{Teoremi di media e formole di maggiorazione relative alle funzioni biarmoniche}, Rend. Sem. Mat. Univ. Padova, 59, 1978, 285--294.

\bibitem{Folland}G. Folland; \emph{How to integrate a polynomial over a sphere}, Amer. Math. Monthly 108 (2001) 446–448.

\bibitem{Gazzola.et.al.2010}F. Gazzola, H.-C. Grunau, G. Sweers; \emph{Polyharmonic boundary value problems: positivity preserving and nonlinear higher order elliptic equations in bounded domains}, Springer Science \& Business Media, 2010.



  \bibitem{Kirmani.Jamil.2018}S.K.N. Kirmani, R.N. Jamil; \emph{Optimization of complex geometry using tenth order partial differential equation}, Sci Inquiry Rev. 2018;2(2):24--31.

 \bibitem{Kobayashi.et.al.2011} T. Kobayashi, N. Kawashima, and Y. Ochiai;
 \emph{Image processing by interpolation using polyharmonic function and increase in processing speed}, IEEJ Trans Elec Electron Eng, 6: S1-S6 (2011).

\bibitem{Li.et.al.1999}X.-F., Li, T.-Y. Fan,  Y.-F. Sun;  \emph{A decagonal quasicrystal with a Griffith crack}, Philosophical Magazine A 79, no. 8 (1999): 1943-1952.



\bibitem{Lysik.2011} G. \L{}ysik; \emph{On the mean-value property for polyharmonic functions}, Acta Mathematica Hungarica 133, no. 1-2 (2011),
133-139.

\bibitem{Lysik.2012} G. \L{}ysik; \emph{Mean-value properties of real analytic functions}, Archiv der Mathematik 98, no. 1 (2012), 61-70.

\bibitem{Lysik.2015} G. \L{}ysik; \emph{A characterization of polyharmonic functions}, Acta Mathematica Hungarica 147, no. 2 (2015).

\bibitem{Lysik.2016} G. \L{}ysik; \emph{Higher order Pizzetti’s formulas}, Rendiconti Lincei 27, no. 1 (2016), 105-115.

\bibitem{MPR} J. J. Manfredi, M. Parviainen, and J. D. Rossi, \emph{An asymptotic mean-value characterization of $p$-harmonic functions}, Proc. Amer. Math. Soc., 138 (2010), pp. 881--889.



  \bibitem{Meleshko.2003}V. V. Meleshko; \emph{Selected topics in the history of the two-dimensional biharmonic problem}, Appl. Mech. Rev. 56, no. 1 (2003): 33-85. 


\bibitem{NM} M. Nicolesco; \emph{Recherches sur les fonctions polyharmoniques}, in Annales scientifiques de l'Ecole normale sup\'erieure, vol. 52, pp. 183-220. 1935.

\bibitem{Nicolesco.book.1936} M. Nicolesco; \emph{Les fonctions polyharmoniques,} Act. Scient. Industr. 331, Hermann et Cie Edit., Paris, 1936.

\bibitem{O} J.S. Ovall; \emph{The Laplacian and mean and extreme values}, The American Mathematical Monthly 123, no. 3 (2016), 287-291.

\bibitem{M} M. Picone; \emph{Sulla convergenza delle successioni di funzioni iperarmoniche}, Bulletin math\'ematique de la Soci\'et\'e Roumaine des Sciences 38, no. 2 (1936), 105-112.


\bibitem{Pizzetti.1909} P. Pizzetti; \emph{Sulla media dei valori che una funzione dei punti dello spazio assume alla superficie di una sfera}, Rend. Lincei 18 (1909), 182-185.

\bibitem{Privaloff} I. Privaloff; \emph{Sur les fonctions harmoniques}, Mat. Sb., 32(3), (1925), 464--471.

\bibitem{Sbr} F. Sbrana; \emph{Sopra una propriet\`a caratteristica delle funzioni poliarmoniche e delle soluzioni dell'equazione delle membrane vibranti}, 
Atti Accad. Naz. Lincei. Rend. VI Ser., 1 (1925), pp. 369--371.

\bibitem{Schmaltz.2014}  C. Schmaltz, P. Peter, M. Mainberger, et al.; \emph{Understanding, Optimising, and Extending Data Compression with Anisotropic Diffusion}, Int. J. Comput. Vis. 108 (2014), pp. 222--240 . 

\bibitem{LZ} L. Zalcman; \emph{Mean-values and differential equations}, Israel Journal of Mathematics 14, no. 4 (1973), 339-352.

\end{thebibliography}
\end{document}